\documentclass[a4paper]{article}

\usepackage{amsmath,amssymb,amsthm}
\usepackage{graphicx}
\usepackage[english]{babel}
\usepackage[utf8]{inputenc}
\usepackage[T1]{fontenc}
\usepackage{mathtools}
\usepackage[margin=3cm]{geometry}
\usepackage{enumitem}
\usepackage{hyperref}
\allowdisplaybreaks

\newcommand{\N}{\mathbb{N}}
\newcommand{\Z}{\mathbb{Z}}

\newcommand{\R}{\mathbb{R}}

\newcommand{\E}{\mathbb{E}}
\newcommand{\ud}{\,\mathrm{d}}
\newcommand{\dd}{\mathsf{d}}
\newcommand{\ZZ}{\mathcal{Z}}
\newcommand{\PP}{\mathbb{P}}

\DeclareMathOperator{\dist}{dist}

\DeclareMathOperator{\Cov}{Cov}
\DeclareMathOperator{\Var}{Var}

\numberwithin{equation}{section}

\newtheorem{theorem}{Theorem}[section]

\newtheorem{observation}[theorem]{Observation}
\newtheorem{lemma}[theorem]{Lemma}
\theoremstyle{remark}
\newtheorem{remark}[theorem]{Remark}

\begin{document}

\title{The maximum of the four-dimensional membrane model}
\author{Florian Schweiger\footnote{Institut für angewandte Mathematik, Universität Bonn, Endenicher Allee 60, 53115 Bonn, Germany, E-Mail: \texttt{schweiger@iam.uni-bonn.de}}}
\date{\today}
\maketitle
\begin{abstract}
We show that the centred maximum of the four-dimensional membrane model on a box of sidelength $N$ converges in distribution. To do so we use a criterion of Ding, Roy and Zeitouni and prove sharp estimates for the Green's function of the discrete Bilaplacian. These estimates are the main contribution of this work and might also be of independent interest. To derive them we use estimates for the approximation quality of finite difference schemes as well as results for the Green's function of the continuous Bilaplacian.
\end{abstract}

\section{Introduction}
\subsection{The membrane model}
A stochastic interface model on a finite subset $A$ of the $\dd$-dimensional lattice $\Z^\dd$ is a probability distribution on height functions $\varphi\colon A\to\R$. The most prominent example of such an interface model is probably the gradient model, also called the discrete Gaussian free field. We define it as the centred Gaussian measure on functions $\{\psi_v\colon v\in A\}$ that are zero outside $A$ given by
\[\PP_A^\nabla(\mathrm{d}\psi)=\frac1{Z_A^\nabla}\exp\left(-\frac12\sum_{v\in \Z^\dd}|\nabla_1\psi_v|^2\right)\prod_{v\in A}\mathrm{d}\psi_v\prod_{v\in\Z^\dd\setminus A}\delta_0(\mathrm{d}\psi_v)
\]
where $\nabla_1\psi_v:=(D^1_i\psi_v)_{i=1}^\dd:=(\psi_{v+e_i}-\psi_v)_{i=1}^\dd$ is the discrete gradient, the vector of discrete forward derivatives. The focus of this work, however, will be on a slightly different model, the so-called membrane model. This is the centred Gaussian measure on functions $\{\psi_v\colon v\in A\}$ that are zero outside $A$ given by
\[\PP_A^\Delta(\mathrm{d}\psi)=\frac1{Z_A^\Delta}\exp\left(-\frac12\sum_{v\in \Z^\dd}|\Delta_1\psi_v|^2\right)\prod_{v\in A}\mathrm{d}\psi_v\prod_{v\in\Z^\dd\setminus A}\delta_0(\mathrm{d}\psi_v)\]
where $\Delta_1 \psi_v:=\sum_{i=1}^\dd\psi_{v+e_i}-2\psi_v+\psi_{v-e_i}$ is the discrete Laplacian.

We caution the reader that there are different normalizations of the gradient and membrane model in the literature. Our definitions are most natural from a PDE point of view. They yield fields that are by a factor of $\frac{1}{\sqrt{2\dd}}$ in the case of the gradient model, or $\frac{1}{2\dd}$ in the case of the membrane model smaller than the fields as defined in \cite{Bolthausen2001} and \cite{Kurt2009}, respectively. In the following, when quoting results from these and other works, we will transform them to our scaling.

We will mostly consider these fields on a box $V_N:=[0,N]^\dd\cap\Z^\dd$ of sidelength $N$. We will denote by $(\psi^\nabla_{N,v})_{v\in V_N}$ and $(\psi^\Delta_{N,v})_{v\in V_N}$ random variables distributed according to $\PP_A^\nabla$ and $\PP_A^\Delta$, respectively, where $A=V_N$. 

A general heuristic is that the $\dd$-dimensional membrane model behaves like the $\frac{\dd}{2}$-dimensional gradient model. In particular, the critical dimension (where covariances decay logarithmically) is $\dd=2$ for the gradient and $\dd=4$ for the membrane model. One interesting question about these models is how their maximums $M^\nabla_N=\max_{v\in V_N}\psi^\nabla_{N,v}$ and $M^\Delta_N=\max_{v\in V_N}\psi^\Delta_{N,v}$ behave as $N$ tends to infinity. The answer depends very much on the dimension. In the supercritical case ($\dd\ge3$ for the gradient model, $\dd\ge5$ for the membrane model) the correlations decay rapidly. Using Stein's method, it was shown in \cite{Chiarini2016b, Chiarini2016} that $M_N^\nabla$ behaves as if the $(\psi^\nabla_{N,v})_{v\in V_N}$ were independent, i.e. that \[\frac{\sqrt{2\dd\log N}}{\sqrt{g^\nabla_\dd}}\left(M_N^\nabla-\sqrt{2dg^\nabla_\dd\log N}+\frac{\sqrt{g^\nabla_\dd}\left(\log(\dd\log N)+\log4\pi\right)}{\sqrt{8\dd\log N}}\right)\] converges in distribution to a Gumbel random variable, where $v_N$ is a lattice point closest to the centre of $[0,N]^\dd$ and $g^\nabla_\dd=\lim_{N\to\infty}\Var(\psi^\nabla_{N,v_N})$; and that the analogous statement holds true for $M^\Delta_N$. In the subcritical cases ($\dd=1$ for the gradient model, $1\le \dd\le3$ for the membrane model) we have that $\frac{M^\nabla_N}{N^{\frac{2-\dd}{2}}}$ and $\frac{M^\Delta_N}{N^{\frac{4-\dd}{2}}}$ converge in distribution, which follows from the fact that the whole rescaled field converges weakly in $C^0$. This is classical for the gradient model, and for the membrane model it was shown for $\dd=1$ in \cite{Caravenna2009} and recently for $2\le \dd\le3$ in \cite{Cipriani2019}.
The most interesting and most subtle case is the critical one ($\dd=2$ for the gradient model, $\dd=4$ for the membrane model). For the gradient model, in a series of papers \cite{Bolthausen2001, Bolthausen2011, Bramson2012, Bramson2016} it was shown that $M^\nabla_N-m^\nabla_N$ converges in distribution to a randomly shifted Gumbel variable, where $m^\nabla_N=\sqrt{\frac{2}{\pi}}\log N-\frac{3}{\sqrt{32\pi}}\log\log N$. Even more is known, in particular convergence of the full extremal process \cite{Biskup2016, Biskup2018}. For the membrane model the picture is less clear. The best previous result \cite{Kurt2009} is that $\frac{M^\Delta_N}{\log N}$ converges to $\frac{1}{\pi}$ in probability. The question whether a centred version of $M^\Delta_N$ converges in distribution was posed for example in \cite{Roy2016,Cipriani2019}. We prove that this is the case.

\begin{theorem}\label{t:convmaxmembrane}
Let $\dd=4$. The random variable \[M^\Delta_N-m^\Delta_N:=M_N^\Delta-\frac{1}{\pi}\log N+\frac{3}{16\pi}\log\log N\] converges in distribution. The limit law is a randomly shifted Gumbel distribution $\mu_\infty$, given by \[\mu_\infty((-\infty,x])=\E e^{-\gamma^*\ZZ e^{-8\pi x} }\ \forall x\] where $\gamma^*$ is a constant and $\ZZ$ is a positive random variable that is the limit in law of \[\ZZ_N = \sqrt{8}\sum_{v\in V_N} (\log N-\pi\psi_{N,v}) e^{-8(\log N-\pi\psi_{N, v})}\,.\]
\end{theorem}
Before we put this result in context and discuss our proof strategy let us point out a generalization.
\begin{remark}
Our approach is not limited to the membrane model. In fact, consider for $l\in \N^+$ the $\nabla^l$-model, given by the probability measure
\begin{align*}
\PP^{(l)}_A(\mathrm{d}\psi)
=\begin{cases}\frac1{Z_A^{(l)}}\exp\left(-\frac12\sum_{v\in \Z^\dd}|\Delta_1^{\frac l2}\psi_v|^2\right)\prod_{v\in A}\mathrm{d}\psi_v\prod_{v\in\Z^\dd\setminus A}\delta_0(\mathrm{d}\psi_v)&l\text{ even}\\\frac1{Z_A^{(l)}}\exp\left(-\frac12\sum_{v\in \Z^\dd}|\nabla_1\Delta_1^{\frac{l-1}{2}}\varphi_v|^2\right)\prod_{v\in A}\mathrm{d}\psi_v\prod_{v\in\Z^\dd\setminus A}\delta_0(\mathrm{d}\psi_v)&l\text{ odd}\end{cases}
\end{align*}
(note that $l=1$ corresponds to the gradient model and $l=2$ to the membrane model) in the critical dimension $\dd=2l$ on the cube $A=[0,N]^\dd\cap\Z^\dd$. Then Theorem \ref{t:convmaxmembrane} generalizes to this setting, and the maximum of the field, appropriately centred, converges in law to a randomly shifted Gumbel distribution. Our proof in the following would only require minor modifications to yield this more general result. However, since the case $l=1$ is covered by \cite{Bramson2016}, while the $\nabla^l$-model for $l>2$ is rarely studied, we choose to focus on the case $l=2$ in the following. This allows us to avoid more complicated notation.
\end{remark}

\subsection{Log-correlated fields}

In recent years there has been great interest in the study of log-correlated Gaussian fields. Very roughly speaking, these are fields where the covariance between the values at two different sites decays logarithmically in their distance. Examples include the two-dimensional gradient and four-dimensional membrane model. It is conjectured that these form a universality class in the sense that many of their features do not depend on the detailed structure of the covariance function (see \cite{Duplantier2014} for a general discussion). One example of such a feature is the behaviour of the maximum of the field, and one expects that convergence in law of the recentred maximum holds true for general log-correlated fields. However, it is a challenging problem to verify this fact for specific examples of log-correlated fields. In recent years convergence in law of the recentred maximum has been proven for the gradient model, as already discussed, and also for various other models. Let us mention branching Brownian motion \cite{Bramson1983}, branching random walks \cite{Aidekon2013}, and also problems from random matrix theory (see \cite{Chhaibi2018} for partial results).

Furthermore there have been efforts to give sufficient criteria for convergence in law of the maximum that cover a wide range of log-correlated fields. In \cite{Madaule2015} this was done for so-called $*$-scale invariant models. Most importantly for us, in \cite{Ding2017} Ding, Roy and Zeitouni gave a set of four assumptions that ensure that the maximum of a field converges in distribution. Let us recall their result, slightly reformulated (we have changed the domain from $[0,N-1]^\dd$ to $[0,N]^\dd$, and replaced $\log_+|a|$ with $\log(1+|a|)$ in \ref{A.0} and \ref{A.1}, but it is straightforward to check that the theorem stated here is equivalent to the theorem as stated in \cite{Ding2017}). We write $d_N(v):=\dist(v,\partial[0,N]^\dd)$ for the distance of $v$ to the boundary of $[0,N]^\dd$ and $d(x):=d_1(x)$.

\begin{theorem}[{\cite[Theorem 1.3 and Theorem 1.4]{Ding2017}}]\label{t:logcorrconvinlaw}

Let $V_N=[0,N]^\dd\cap\Z^\dd$, and let $\varphi_N=\{\varphi_{N,v}\colon v\in V_N\}$ be a centred Gaussian field. Assume that
\begin{enumerate}[label=(A.\arabic*),ref=(A.\arabic*),start=0]
	\item\label{A.0} \textbf{(Logarithmically  bounded fields)} There is a constant $\alpha_0 >0$ such that for all $u,v\in V_N$,
\[\Var \varphi_{N,v}\le \log N+\alpha_0\] and 
\[\E(\varphi_{N,v}-\varphi_{N,u})^2  \le 2 \log(1+|u-v|)-|\Var\varphi_{N, v}-\Var \varphi_{N, u}|+4\alpha_0\,.\]
	\item\label{A.1} \textbf{(Logarithmically correlated fields)} For any $\delta>0$ there is a constant $\alpha^{(\delta)}>0$ such that for all $u, v\in V_N$ with $\min(d_N(u),d_N(v))\ge\delta N$ \[|\Cov(\varphi_{N, v},\varphi_{N, u})-(\log N-\log(1+|u-v|))| \le \alpha^{(\delta)}\,.\]
	\item\label{A.2} \textbf{(Near diagonal behaviour)} There are both a continuous function $f_1\colon (0,1)^\dd \to\R$ and a function $f_2\colon \Z^\dd \times \Z^\dd\to \R$ such that the following holds. For all $L, \varepsilon,\delta > 0$, there exists $N_0= N_0(L,\varepsilon, \delta)$ such that for all $x\in [0,1]^\dd$, $N\geq N_0$  such that $Nx\in\Z^\dd$ and $d(x)\ge\delta$, and for all $u, v \in [0,L]^\dd\cap\Z^\dd $ we have \[|\Cov( \varphi_{N, Nx+v},\varphi_{N, Nx+u})-\log N -f_1(x) -f_2(u,v)|< \varepsilon\,.\]
	\item\label{A.3} \textbf{(Off diagonal behaviour)} There is a continuous function $f_3\colon \mathcal{D}^\dd\to \R$, where $\mathcal {D}^\dd=\{(x,y): x,y\in (0,1)^\dd, x\neq y\}$ such that the following holds. For all $L, \varepsilon,\delta > 0$ there exists $ N_1= N_1(L,\varepsilon, \delta)>0$ such that for all $x, y\in [0,1]^\dd$, $N \ge N_1 $ such that $Nx,Ny\in\Z^\dd$, $\min(d(x),d(y))\ge\delta$ and $|x-y| \ge \frac{1}{L}$ we have
\[|\Cov(\varphi_{N,Nx},\varphi_{N,Ny})-f_3(x,y)|<\varepsilon\,.\]
\end{enumerate}
Let $M_N = \max_{v\in V_N} \varphi_{N, v}$ and 
\[m_N = \sqrt{2\dd} \log N - \frac{3}{2\sqrt{2\dd}} \log \log N\,.\]
Then the sequence $M_N-m_N$ converges in distribution to a randomly shifted Gumbel distribution $\mu_\infty$. The limit distribution is given by \[\mu_\infty((-\infty,x])=\E e^{-\gamma^*\ZZ e^{-\sqrt{2\dd}x} }\ \forall x\] where $\gamma^*$ is a constant and $\ZZ$ is a positive random variable that is the limit in law of \[\mathcal Z_N = \sum_{v\in V_N} (\sqrt{2\dd} \log N - \varphi_{N, v}) e^{-\sqrt{2\dd}(\sqrt{2\dd} \log N - \varphi_{N, v})}\,.\]

\end{theorem}
This theorem easily implies Theorem \ref{t:convmaxmembrane} once we show that $\psi^\Delta_N$ (or rather $\sqrt{8}\pi\psi^\Delta_N$) satisfies assumptions \ref{A.0}, \ref{A.1}, \ref{A.2}, \ref{A.3}. In fact we can prove even slightly stronger statements than these.
Let us state the precise results that we will prove. We abbreviate $\lambda=\sqrt{8}\pi$. 
\begin{theorem}\label{t:membranemodifiedassumptions}
The field $\varphi_N:=\lambda\psi^\Delta_N$ in dimension $\dd=4$ satisfies
\begin{enumerate}[label=(A.\arabic*'),ref=(A.\arabic*'),start=0]
	\item\label{A.0'} There is a constant $\alpha_0' >0$ such that for all $u,v\in V_N$,
\[\Var \varphi_{N, v} \le \min\left(\log N+ \alpha_0',\alpha_0'\log (2+d_N(v))\right)\] and 
\[\Var \varphi_{N, v}-\Cov(\varphi_{N, v},\varphi_{N, u})  \le \log(1+ |u-v|) + 2\alpha_0'\,.\]
	\item\label{A.1'} There is a constant $\alpha_0''>0$ such that for all $u, v\in V_N$ \[\left|\Cov( \varphi_{N, v},\varphi_{N, u}) - \log\left(2+\frac{\max(d_N(u),d_N(v))}{1+|u-v|}\right) \right| \le \alpha_0''\,.\]
	\item\label{A.2'} There are a constant $\theta_0>0$, a continuous function $f_1\colon(0,1)^4 \to\R$ and a function $f_2\colon \Z^4 \times \Z^4 	\to \R$ such that the following holds. For all $L, \varepsilon> 0$, $\theta>\theta_0$ there exists $N_0'= N_0'(L,\varepsilon, \theta)$ such that for all $x\in [0,1]^4$, $N\geq N_0'$  such that $Nx\in\Z^4$ and $d(x)\ge \frac{(\log N)^\theta}{N}$, and for all $u, v \in [0,L]^4 \cap\Z^4$ we have \[|\Cov( \varphi_{N, Nx+v},\varphi_{N, Nx+u})-\log N -f_1(x) -f_2(u,v)|< \varepsilon\,.\]
	\item\label{A.3'} There are a constant $\theta_1>0$ and a continuous function $f_3\colon \mathcal{D}^4\to \R$, where $\mathcal {D}^4=\{(x,y): x,y\in (0,1)^4, x\neq y\}$ such that the following holds. For all $L, \varepsilon>0$, $\theta>\theta_1$ there exists $ N_1'= N_1'(L,\varepsilon, \theta)$ such that for all $x, y\in V$, $N \ge N_1' $ such that $Nx,Ny\in\Z^4$, $\min(d(x),d(y))\ge \frac{(\log N)^\theta}{N}$ and $|x-y| \ge \frac{1}{L}$ we have
\[|\Cov( \varphi_{N, Nx},\varphi_{N, Ny})-f_3(x,y)|< \varepsilon\,.\]
\end{enumerate}
\end{theorem}
It is not hard to check that the assumptions \ref{A.0'}, \ref{A.1'}, \ref{A.2'}, \ref{A.3'} imply \ref{A.0}, \ref{A.1}, \ref{A.2}, \ref{A.3} respectively, so that Theorem \ref{t:convmaxmembrane} is a straightforward corollary of Theorem \ref{t:membranemodifiedassumptions}. We give a few more details in Section \ref{s:proofmainthm}.

The proof of Theorem \ref{t:membranemodifiedassumptions} is the main contribution of this work. In the next section we will describe our approach.

\subsection{Green's function estimates}\label{s:introgreen}

The covariance function of the membrane model is the Green's function $G^\Delta_N$ of the discrete Bilaplacian on the grid $[0,N]^\dd$ with zero boundary data, and the assumptions \ref{A.0'}, \ref{A.1'}, \ref{A.2'}, \ref{A.3'} all correspond to certain estimates for this Green's function. Therefore our goal is to understand this Green's function. We are going to apply tools from PDE theory and numerical analysis, so before proceeding further it is convenient to rescale our domain to a unit box. Let $h=\frac{1}{N}$, let $V_h=[0,1]^4\cap(h\Z)^4$, and let $\psi^\Delta_{h,x}:=\psi^\Delta_{N,\frac xh}$. Let $G^\Delta_N$ and $G^\Delta_h$ be the covariance functions of $\psi^\Delta_N$ and $\psi^\Delta_h$. Then also $G^\Delta_h(x,y)=G^\Delta_N\left(\frac xh,\frac yh\right)$.

Using $G^\Delta_N$ and $G^\Delta_h$, $V_N$ and $V_h$, and $\psi_N$ and $\psi_h$ simultaneously is a slight abuse of notation. It should, however, always be clear from the context which object we are referring to. Let us also remark that from a PDE point of view it would arguably be more natural to choose $h=\frac{1}{N+2}$ and rescale $[0,N]^4$ to $[h,1-h]^4$, as this would give our domain a natural boundary layer of zeros, matching the continuous Dirichlet boundary data. Our choice of rescaling, however, is in line with \cite{Ding2017}.

\begin{observation}\label{o:estimatesGh}
Under the aforementioned rescaling, each statement \ref{A.0'}, \ref{A.1'}, \ref{A.2'}, \ref{A.3'} from Theorem \ref{t:membranemodifiedassumptions} for $\lambda\psi_N^\Delta$ in dimension $\dd=4$ is equivalent to the corresponding following statement for $G^\Delta_h$.
\begin{enumerate}[label=(B.\arabic*'),ref=(B.\arabic*'),start=0]
	\item\label{B.0'} There is a constant $\alpha_0' >0$ such that for all $x,y\in V_h$,
 \[\lambda^2G^\Delta_h(x,x) \le \min\left(-\log h+ \alpha_0',\alpha_0'\log\left(2+\frac{d(x)}{h}\right)\right)\] and
\[\lambda^2\left(G^\Delta_h(x,x)-G^\Delta_h(x,y)\right)\le \log\left(1+\frac{|x-y|}{h}\right)+2\alpha_0'\,.\]
	\item\label{B.1'} There is a constant $\alpha_0'' >0$ such that for all $x,y\in V_h$ \[\left|\lambda^2 G^\Delta_h(x,y) -\log\left(2+\frac{\max(d(x),d(y))}{h+|x-y|}\right) \right| \le \alpha_0''\,.\]
	\item\label{B.2'} There are a constant $\theta_0>0$, a continuous function $f_1\colon(0,1)^4 \to\R$ and a function $f_2\colon \Z^4 \times \Z^4 	\to \R$ such that the following holds. For all $L, \varepsilon> 0$, $\theta>\theta_0$ there exists $N_0'= N_0'(L,\varepsilon, \theta)$ such that for all $h\le\frac1{N_0'}$ with $\frac1h\in\N$, all $x\in V_h$ such that $d(x)\ge h|\log h|^\theta$ and for all $u, v \in [0,L]^4\cap\Z^4 $ we have  \[\left|\lambda^2 G^\Delta_h(x+hu,x+hv)+\log h -f_1(x) -f_2(u,v)\right|< \varepsilon\,.\]
	\item\label{B.3'} There are a constant $\theta_1>0$ and a continuous function $f_3\colon \mathcal{D}^4\to \R$, where $\mathcal {D}^4=\{(x,y): x,y\in (0,1)^4, x\neq y\}$ such that the following holds. For all $L, \varepsilon>0$, $\theta>\theta_1$ there exists $ N_1'= N_1'(L,\varepsilon, \theta)$ such that for all $h\le\frac1{N_1'}$ with $\frac1h\in\N$ and for $x, y\in V_h$ such that $\min(d(x),d(y))\ge h|\log h|^\theta$ and $|x-y| \ge \frac{1}{L}$ we have
\[\left|\lambda^2G^\Delta_h(x,y)-f_3(x,y)\right|<\varepsilon\,.\]
\end{enumerate}
\end{observation}
Let us discuss how one might prove Theorem \ref{t:membranemodifiedassumptions}, or rather the statements \ref{B.0'}, \ref{B.1'}, \ref{B.2'}, \ref{B.3'}. We write $\Gamma_h=(h\Z)^4\cap\left([-h,1+h]^4\setminus[0,h]^4\right)$. The function $G^\Delta_h$ is the Green's function associated to the discrete boundary value problem
\begin{alignat}{2}
		\Delta_h^2u_h&=f_h\qquad && \text{in }V_h\nonumber\\
		u_h&=0 \qquad && \text{on }\Gamma_h\nonumber\\
		D^h_\nu u_h&=0\qquad && \text{on }\Gamma_h\label{eq:bilaplace}
\end{alignat}
(where $D^h_\nu u(x)=\frac{u(x+h\nu)-u(x)}{h}$ and $\nu$ is an outward unit normal vector). That is, for $y\in V_h$ the function $G_h(\cdot,y)$ is the unique solution of that equation with right hand side $f_h=\delta_h(y)$, defined as $\delta_{h,y}(x)=\begin{cases}\frac1{h^4}&\text{if }x=y\\0 &\text{otherwise}\end{cases}$.

One previous strategy to prove estimates for $G^\Delta_h$, introduced in \cite{Kurt2009} and used as well in \cite{Cipriani2013}, was to compare $G^\Delta_h$ to $\overline G^\Delta_h$, the Green's function associated to the discrete boundary value problem
\begin{alignat}{2}
		\Delta_h^2u_h&=f_h\qquad && \text{in }V_h\nonumber\\
		u_h&=0 \qquad && \text{on }\Gamma_h'\nonumber\\
		\Delta_h u&=0\qquad && \text{on }\Gamma_h\label{eq:bilaplace_otherbc}
\end{alignat}
where $\Gamma'_h=(h\Z)^4\cap\left([-2h,1+2h]^4\setminus[-h,1+h]^4\right)$. The problem \eqref{eq:bilaplace_otherbc} can be seen as an iterated version of the discrete Poisson problem, and so many of the analytic and probabilistic tools available for the latter also have a version for \eqref{eq:bilaplace_otherbc}. In particular, there are random walk representations for $\overline G^\Delta_h$ that allow to control it well. The strategy in \cite{Kurt2009} then was to use PDE techniques to compare solutions of \eqref{eq:bilaplace} and \eqref{eq:bilaplace_otherbc}. This allows to estimate the difference between $G_h$ and $\overline G_h$ uniformly in compact subsets of $(0,1)^4$. For our purposes, this is not good enough, as for \ref{B.2'} and \ref{B.3'} an error term that is only bounded is already too much. Note however that results similar to \ref{B.0'}, \ref{B.1'} can be proved using these methods. In fact, \cite[Proposition 1.1]{Kurt2009} and \cite[Lemma 2.1]{Cipriani2013} are already weaker versions of \ref{B.0'} and \ref{B.1'}.

In \cite{Muller2019} the authors considered $G^\Delta_h$ in dimensions 2 and 3, and used a very different strategy. They used a compactness argument to transfer estimates for the continuous Green's function in domains with singularities to the discrete setting. This allowed them the prove discrete Caccioppoli inequalities (i.e. $L^2$-based decay estimates on balls of various sizes) and to conclude from these estimates for $G^\Delta_h$. In principle, this strategy can also be applied in our four-dimensional setting. One obstacle to this is that, unlike the two- or three-dimensional case, the relevant continuous estimates cannot be found in the literature. Even more importantly, the estimates in \cite{Muller2019} are all up to a possibly large constant, and so the argument would have to be modified significantly to obtain estimates such as \ref{B.2'} and \ref{B.3'}.

\medskip

Instead of the aforementioned approaches to derive estimates for $G^\Delta_h$ we will use estimates for the approximation quality of finite difference schemes for the Bilaplacian. This idea is not completely new, as for example in \cite{Cipriani2019} estimates for finite difference schemes from \cite{Thomee1964} were used to prove convergence of the rescaled four-dimensional membrane model in some negative Sobolev space. However, we would like to obtain a much stronger conclusion, namely pointwise estimates for the difference of the discrete and continuous Green's function. The result from \cite{Thomee1964} is very general, but because of its generality it requires in our specific case very strong assumptions on the solution of the continuous Bilaplace equation to be approximated (being $C^5$) to yield estimates useful for us (the $W^{2,2}_h$-approximation error decaying like $h^\frac12$).

We will use a rather different estimate for the approximation quality of finite difference schemes. 
We will discuss the details in Section \ref{s:finitediff}. Roughly speaking, the result is the following: Let $2<s<\frac52$, let $u\in W^{s,2}\cap W^{2,2}_0((0,1)^4)$ extended by 0 to $\R^4$, and assume that $\Delta^2u=f$ in $(0,1)^4$, so that $u$ satisfies
\begin{alignat}{2}
		\Delta^2u&=f\qquad && \text{in }(0,1)^4\nonumber\\
		u&=0 \qquad && \text{on }\partial(0,1)^4\nonumber\\
		\partial_\nu u&=0\qquad && \text{on }\partial(0,1)^4\,.\label{eq:bilaplace_cont}
\end{alignat}
Furthermore, let $u_h\colon(h\Z)^4\to\R$ be the solution of
\begin{alignat*}{2}
		\Delta_h^2u_h&=T^{h,3,3,3,3}f\qquad && \text{in }V_h\\
		u_h& =0 \qquad &&\text{on }(h\Z)^4\setminus V_h
\end{alignat*}
where $T^{h,3,3,3,3}$ is a certain regularization operator. Then
\[\|u-u_h\|_{W^{2,2}_h(V_h)}\le Ch^{s-2}\|u\|_{W^{s,2}((0,1)^4)}\]
where $\|\cdot\|_{W^{2,2}_h(V_h)}$ is a discrete Sobolev norm.

This result is inspired by closely related recent results in \cite{Muller2019b}. However, in that work the focus is on obtaining estimates as above for $s$ as large as possible. In the case of interest to us, $s<\frac52$, the result can essentially be shown using the methods from \cite{Gavrilyuk1983,Ivanovich1986,Jovanovic2014}.

We will use this result to compare solutions of \eqref{eq:bilaplace} with solutions of \eqref{eq:bilaplace_cont}. In particular, we will use it when $u$ is the regular part of the continuous Green's function on $[0,1]^4$. 
To do so, we need regularity estimates for solutions of \eqref{eq:bilaplace_cont}. As already mentioned, optimal estimates for higher order elliptic problems on four-dimensional polyhedral domains are not yet in the literature. Instead we will use much weaker estimates (similar to ones in \cite{Mitrea2013, Mayboroda2014}) which are nonetheless sharp enough for our purposes. These estimates will allow us to place the regular part of the Green's function in $W^{2+\kappa_0,2}$ for some small $\kappa_0>0$, and this is good enough to apply the estimate above.

We will also need to have good estimates for the discrete Green's function on the full space $(h\Z)^4$. These were derived in \cite{Mangad1967} using Fourier analysis. Furthermore, Theorem \ref{t:mss} gives us control over the $W^{2,2}_h$-norm of the difference of $u$ and $u_h$, while we are actually interested in the $L^\infty_h$-norm and want it to decay. To achieve this, we will use a discrete Sobolev-inequality that allows us to control the $L^\infty_h$-norm by the $W^{2,2}_h$-norm at the cost of a term logarithmic in $h$. The presence of this term is the reason why we can prove \ref{B.2'} and \ref{B.3'} only up to distance $|\log h|^\theta$ to the boundary. For \ref{B.0'} and \ref{B.1'} we do not need a decaying but only a bounded error term and so we can prove these estimates on the whole domain.

\medskip

We will give the details of the argument that we sketched here in the following sections. In Section \ref{s:prelim} we gather various useful results: The aforementioned result on finite difference schemes from \cite{Muller2019b}, as well as some discrete inequality of Poincar\'e-Sobolev-type. These tools will allow us to compare $G^\Delta_h$ with various other Green's functions: the discrete Green's function of the full space (that we discuss in Section \ref{s:estgreendisc}) and the continuous Green's functions of the box $[0,1]^4$ and of the full space (that we both discuss in Section \ref{s:estgreencont}). After all these preparations we can then turn to the proof of Theorem \ref{t:membranemodifiedassumptions} in Section \ref{s:proofmainthm}. We first prove a crucial lemma, Lemma \ref{l:discandcontGreenareclose} that shows that the regular part of the discrete and continuous Green's functions on the box are uniformly close, and then we use this Lemma and the results of the preceding sections to establish Theorem \ref{t:membranemodifiedassumptions}. Finally we use Theorem \ref{t:logcorrconvinlaw} to conclude Theorem \ref{t:convmaxmembrane} as well.

\subsection{Notation}
Our notation mostly follows that of \cite{Muller2019}, with some minor modifications. From now on we will only consider the membrane and not the gradient model, so there is no risk of confusion when we drop all superscripts $\Delta$.

In the following $C$ denotes a constant that is independent of all other occurring variables, but whose precise value may change from occurrrence to occurrence. By $C_{r,s,t,\ldots}$ we similarly denote a constant depending only on $r,s,t,\ldots$ whose precise value may change from occurrence to occurrence. Occasionally we write $r=s+O(t)$ to express $|r-s|\le Ct$.

We write $\partial_i$ for the partial derivative in direction $e_i$, and $\partial^\alpha=\partial_1^{\alpha_1}\ldots\partial_4^{\alpha_4}$ for a multi-index $\alpha$. We denote by $\nabla,\nabla^2,\Delta,\Delta^2$ the gradient, the Hessian matrix, the Laplacian and the Bilaplacian respectively. In particular the reader should not confuse $\nabla^2$ and $\Delta$. For $\Omega\subset\R^4$ open, $k\in\N$, $p\in[1,\infty]$, $\alpha\in(0,1)$ we use the $L^p$-space $L^p(\Omega)$, the H\"older space $C^{0,\alpha}(\Omega)$ and the Sobolev space $W^{k,2}(\Omega)$; the latter equipped with the norm $\|u\|^2_{W^{k,2}(\Omega)}=\sum_{|\alpha|\le k}\|\partial^\alpha u\|^2_{L^2(\Omega)}$. For $s>0$ not an integer (i.e. $s=k+t$ where $k\in\N$, $0<t<1$) we will also encounter the fractional Sobolev space $W^{s,2}(\Omega)$ with norm $\|u\|^2_{W^{s,2}(\Omega)}=\|u\|^2_{W^{k,2}(\Omega)}+[u]^2_{W^{s,2}(\Omega)}$ and the seminorm $[u]^2_{W^{s,2}(\Omega)}=\sum_{|\alpha|=k}\int_\Omega\int_\Omega\frac{\left|\partial^\alpha u(x)-\partial^\alpha u(y)\right|^2}{|x-y|^{4+2t}}\ud x\ud y$. For any $s<0$ we define $W^{s,2}(\Omega)$ as the dual of $W^{-s,2}_0(\Omega)$. We extend these definitions to vector-valued functions by taking the $l^2$-norm of the norms of the components.

By $B_r(x)$ we denote the open ball of radius $r$ around $x$.

For a unit vector $a\in\R^4$ define the forward difference quotient $D^h_av(x):=\frac1h(v(x+ha)-v(x))$ and the backward difference quotient $D^h_{-a}v(x):=\frac1h(v(x)-v(x-ha))$. When $a$ is a standard unit vector $e_i$, we write $D^h_i$ instead of $D^h_{e_i}$ and $D^h_{-i}$ instead of $D^h_{-e_i}$.

The discrete gradient is the vector $\nabla_hv(x):=(D^h_iv(x))_{i=1}^4$, the discrete Hessian is the tuple $\nabla^2_hv(x):=(D^h_iD^h_{-j}v(x))_{i,j=1}^4$, the discrete Laplacian is $\Delta_hv(x):=\sum_{i=1}^4D^h_iD^h_{-i}v(x)$, and the discrete Bilaplacian is $\Delta_h^2:=\Delta_h\circ\Delta_h$. For a multi-index $\alpha\in\N^4$ we write $D^h_\alpha u_h(x)=(D^h_1)^{\alpha_1}\ldots(D^h_4)^{\alpha_4}u_h(x)$.
Given $A\subset (h\Z)^4$ and $u_h\colon A\to\R$, we define $\|u_h\|_{L^2_h(A)}^2=\sum_{x\in A}h^4|u_h(x)|^2$, and $\|u_h\|_{L^\infty_h(A)}=\sup_{x\in A}|u_h(x)|$. We will also use the discrete Sobolev-norm $\|u_h\|_{W^{2,2}_h(A)}^2=\|u_h\|_{L^2_h(A)}^2+\|\nabla_hu_h\|_{L^2_h(A)}^2+\|\nabla_h^2u_h\|_{L^2_h(A)}^2$, where we extend the definitions to vector-valued functions as before.

For $r>0$ and $x\in(h\Z)^4$ we let $Q^h_r(x)=x+[-r,r]^4\cap(h\Z)^4$ be the cube of diameter $2r$ around $x$. 

Let us also fix once and for all a smooth function $\eta\colon\R^4\to\R$ that is equal to 1 on $B_{\frac12}(0)$ and 0 outside $B_1(0)$. We define $\eta^{(r)}(x)=\eta(rx)$, $\eta_y^{(r)}(x)=\eta^{(r)}(x-y)$ and let $\eta^{(r)}_{h,y}$ be the restriction of $\eta_y^{(r)}$ to $(h\Z)^4$. Thus $\eta^{(r)}_y$ and $\eta^{(r)}_{h,y}$ are cut-off functions at scale $r$ around $y$.

\section{Preliminaries}\label{s:prelim}

\subsection{Discrete Inequalities}\label{s:estpoincsob}
We collect here two discrete inequalities that we will use several times in the following. We begin with a Poincar\'e inequality.

\begin{lemma}\label{l:poincare}
Let $x_*\in (h\Z)^4$, $r\ge0$. Let $u_h\colon(h\Z)^4\to\R$ and suppose that $u_h$ vanishes on at least one of the faces of $Q_r(x_*)$. Let this face be contained in a plane $x_i=c$. Then 
\begin{align}
\|u_h\|_{L^2_h(Q^h_r(x_*))}^2\le Cr^2\sum_{x\colon\{x,x+he_i\}\subset Q^h_r(x_*)}h^4|D^h_iu_h(x)|^2
\le Cr^2\|\nabla_hu_h\|_{L^2_h(Q^h_r(x_*))}^2\,.\label{eq:poincare1}
\end{align}
\end{lemma}
\begin{proof}
This is a particular case of \cite[Lemma 2.1]{Muller2019}. For the convenience of the reader we give a proof. The second inequality is obvious, so we only prove the first. By translating and reflecting the lattice and renaming the coordinates, we can assume $i=4$, $Q^h_r(x_*)=[0,2r]^4\cap(h\Z)^4$. We write $x=(x',x_4)$ where $x'\in\R^3$, $x_4\in\R$, $u_h=0$ if $x_4=0$. We will prove the one-dimensional estimate
\begin{equation}\label{eq:poincare2}
\sum_{x_4\in[0,2r]\cap h\Z}|u_h(x',x_4)|^2\le Cr^2\sum_{x_4\in[0,2r-h]\cap h\Z}|D^h_4u_h(x',x_4)|^2\,.
\end{equation}
Once we have established this, \eqref{eq:poincare1} follows by multiplying \eqref{eq:poincare2} by $h^4$ and summing over all $x'\in[0,2r]^3\cap(h\Z)^3$. To prove \eqref{eq:poincare2}, we use $u(x',0)=0$ and write
\begin{align*}
	|u_h(x',x_4)|&=\left|\sum_{y_4\in[0,x_4-h]\cap h\Z}u_h(x',y_4+h)-u_h(x',y_4)\right|\\
	&=\left|\sum_{y_4\in[0,x_4-h]\cap h\Z}hD^h_4u_h(x',y_4)\right|\\
	&\le h\left(\frac{x_4}{h}\right)^{\frac12}\left(\sum_{y_4\in[0,x_4-h]\cap h\Z}|D^h_4u_h(x',y_4)|^2\right)^{\frac12}\\
	&\le \sqrt{2hr}\left(\sum_{y_4\in[0,2r-h]\cap h\Z}|D^h_4u_h(x',y_4)|^2\right)^{\frac12}
\end{align*}
and therefore
\begin{align*}
\sum_{x_4\in[0,2r]\cap h\Z}|u_h(x',x_4)|^2\le \frac{2r}{h}2hr\sum_{y_4\in[0,2r-h]\cap h\Z}|D^h_4u_h(x',y_4)|^2\le4r^2\sum_{y_4\in[0,2r-h]\cap h\Z}|D^h_4u_h(x',y_4)|^2\,.
\end{align*}
This shows \eqref{eq:poincare2}.
\end{proof}

Next we give an inequality of Poincar\'e-Sobolev type. Given $u_h\colon(h\Z)^4\to\R$ that vanishes outside of $V_h$ we would like to estimate its pointwise values by the $\|u_h\|_{W^{2,2}_h((h\Z)^4)}$-norm. We cannot hope for such an estimate to hold with a constant independent of $h$, as the (continuous) Sobolev space $W^{2,2}((0,1)^4)$ does not embed into $L^\infty((0,1)^4)$. However, by Strichartz's  \cite{Strichartz1971} version of the Moser-Trudinger inequality any $u\in W^{2,2}((0,1)^4)$ with $\|u\|_{W^{2,2}((0,1)^4)}=1$ satisfies $\int_{(0,1)^4}e^{c|u(x)|^2}\ud x\le C$, and this suggests that $u$ can diverge at worst like $\sqrt{|\log|x||}$. So back in the discrete setting we can hope for an estimate with a factor scaling like $\sqrt{|\log h|}$. Indeed we have the following result:

\begin{lemma}\label{l:poincaresobolev}
Assume that $u_h\colon(h\Z)^4\to\R$ vanishes outside of $V_h$. Then for any $x\in V_h$ we have
\[|u_h(x)|\le C\sqrt{\log\left(2+\frac{d(x)}{h}\right)}\|u_h\|_{W^{2,2}_h((h\Z)^4)}\,.\]
\end{lemma}
This lemma in combination with Theorem \ref{t:mss} will allow us to control the distance between the solution of a continuous Bilaplace equation and its discrete approximation at the cost of a logarithmic divergence (which we will be able to absorb in the applications in Section \ref{s:proofmainthm}).
\begin{proof}[Proof of Lemma \ref{l:poincaresobolev}]
We first want to localize to a ball around $x$. Let $v_h=\eta^{(d(x)+h)}_{h,x}u_h$. Then $v_h(x)=u_h(x)$. Furthermore $v_h$ is supported on $Q^h_{d(x)+h}(x)$. The discrete chain rule implies that
\begin{align*}
	|D^h_iv_h(y)|&\le C\sup_{z\in Q^h_h(y)}\left|D^h_i\eta^{(d(x)+h)}_{h,x}(z)\right|\sup_{z\in Q^h_h(y)}|u_h(z)|\\
	&\quad+C\sup_{z\in Q^h_h(y)}\left|\eta^{(d(x)+h)}_{h,x}(z)\right|\sup_{z\in Q^h_h(y)}|D^h_iu_h(z)|\\
	&\le C\sup_{z\in Q^h_h(y)}\left|D^h_i\eta^{(d(x)+h)}_{h,x}(z)\right|\left(\sum_{z\in Q^h_h(y)}|u_h(z)|^2\right)^{\frac12}\\
	&\quad+C\sup_{z\in Q^h_h(y)}\left|\eta^{(d(x)+h)}_{h,x}(z)\right|\left(\sum_{z\in Q^h_h(y)}|D^h_iu_h(z)|^2\right)^{\frac12}
\end{align*}
and a similar expression for $|D^h_iD^h_{-j}v_h(y)|$. If we sum the squares of these eximates over $y$, we see that
\begin{align}
	\|v_h\|_{W^{2,2}_h((h\Z)^4)}&\le C\|\eta^{(d(x)+h)}_{h,x}\|_{L^\infty_h((h\Z)^4)}\|\nabla_h^2u_h\|_{L^2_h(Q^h_{d(x)+2h}(x))}\nonumber\\
	&\qquad+ C\|\nabla_h\eta^{(d(x)+h)}_{h,x}\|_{L^\infty_h((h\Z)^4)}\|\nabla_hu_h\|_{L^2_h(Q^h_{d(x)+2h}(x))}\nonumber\\
	&\qquad+ C\|\nabla_h^2\eta^{(d(x)+h)}_{h,x}\|_{L^\infty_h((h\Z)^4)}\|u_h\|_{L^2_h(Q^h_{d(x)+2h}(x))}\nonumber\\
	&\quad\le C\|\nabla_h^2u_h\|_{L^2_h(Q^h_{d(x)+2h}(x))}+\frac{C}{d(x)+h}\|\nabla_hu_h\|_{L^2_h(Q^h_{d(x)+2h}(x))}\nonumber\\
	&\qquad+\frac{C}{(d(x)+h)^2}\|u_h\|_{L^2_h(Q^h_{d(x)+2h}(x))}\,.\label{eq:poinsob1}
\end{align}
We can apply Lemma \ref{l:poincare} to $u_h$ and $D^h_iu_h$ for any $i\in\{1,\ldots,4\}$, because these vanish on $Q^h_{d(x)+2h}(x)\setminus[-h,1+h]^4$ and hence in particular on a face of $Q^h_{d(x)+2h}(x)$. Thus we obtain
\begin{align}
\|u_h\|_{L^2_h(Q^h_{d(x)+2h}(x))}&\le C(d(x)+2h)\|\nabla_hu_h\|_{L^2_h(Q^h_{d(x)+2h}(x))}\nonumber\\
&\le C(d(x)+2h)^2\|\nabla_h^2u_h\|_{L^2_h(Q^h_{d(x)+2h}(x))}\,.\label{eq:poinsob2}
\end{align}
If we combine this with \eqref{eq:poinsob1} and note that $d(x)+2h\le2(d(x)+h)$, we obtain
\begin{equation}\label{eq:poinsob6}
\|v_h\|_{W^{2,2}_h((h\Z)^4)}\le C\|u_h\|_{W^{2,2}_h((h\Z)^4)}\,.
\end{equation}
Furthermore, an argument analogous to the one that led to \eqref{eq:poinsob2} shows that
\begin{equation}\label{eq:poinsob3}
	\|v_h\|_{L^2_h((h\Z)^4)}\le C(d(x)+h)^2\|\nabla_h^2v_h\|_{L^2_h((h\Z)^4)}\,.
\end{equation}
Now we are in a position to apply discrete Fourier analysis, similar to the proof of \cite[Proposition B.1]{Kurt2009}. Let
\[\widehat{v_h}(\xi)=h^4\sum_{y\in (h\Z)^4}v_h(y)e^{iy\cdot\xi}\]
for any $\xi\in\left[-\frac{\pi}{h},\frac{\pi}{h}\right]^4$ be the Fourier transform of $v_h$. Then we also have the inverse formula
\[v_h(z)=\frac{1}{(2\pi)^4}\int_{\left[-\frac{\pi}{h},\frac{\pi}{h}\right]^4}\widehat{v_h}(\xi)e^{-iz\cdot\xi}\ud \xi\]
for any $z\in (h\Z)^4$, and Plancherel's formula in the form
\[\int_{\left[-\frac{\pi}{h},\frac{\pi}{h}\right]^4}|\widehat{v_h}(\xi)|^2\ud \xi=(2\pi h)^4\sum_{y\in (h\Z)^4}|v_h(y)|^2=(2\pi)^4\|v_h(y)\|_{L^2_h((h\Z)^4)}^2\,.\]
We have \[\widehat{D^h_\alpha v_h}(\xi)=(e^{-ih\xi_1}-1)^{\alpha_1}\ldots(e^{-ih\xi_4}-1)^{\alpha_4}\widehat{v_h}(\xi)\]
for any $\alpha\in\N^4$. This implies
\[\left|\widehat{D^h_\alpha v_h}(\xi)\right|\ge\frac{1}{C}|\xi_1|^{\alpha_1}\ldots|\xi_4|^{\alpha_4}|\widehat{v_h}(\xi)|\]
for any $\xi\in\left[-\frac{\pi}{h},\frac{\pi}{h}\right]^4$. In combination with Plancherel's formula and \eqref{eq:poinsob3} we conclude
\begin{align}
	\int_{\left[-\frac{\pi}{h},\frac{\pi}{h}\right]^4}|\xi|^4|\widehat{v_h}(\xi)|^2&\le C\|\nabla_h^2v_h\|_{L^2_h((h\Z)^4)}^2\le C\|v_h\|_{W^{2,2}_h((h\Z)^4)}^2\,,\label{eq:poinsob4}\\
	\int_{\left[-\frac{\pi}{h},\frac{\pi}{h}\right]^4}|\widehat{v_h}(\xi)|^2&\le C\|v_h\|_{L^2_h((h\Z)^4)}^2\le C(d(x)+h)^4\|v_h\|_{W^{2,2}_h((h\Z)^4)}^2\,.\label{eq:poinsob5}
\end{align}
Next, we estimate
\begin{align*}
	&|v_h(x)|=\frac{1}{(2\pi)^4}\left|\int_{\left[-\frac{\pi}{h},\frac{\pi}{h}\right]^4}\widehat{v_h}(\xi)e^{-ix\cdot\xi}\ud \xi\right|\\
	&\quad\le C\int_{\left[-\frac{\pi}{h},\frac{\pi}{h}\right]^4}|\widehat{v_h}(\xi)|\ud \xi\\
	&\quad\le C\left(\int_{\left[-\frac{\pi}{h},\frac{\pi}{h}\right]^4}\left(|\xi|^4+\frac{1}{(d(x)+h)^4}\right)|\widehat{v_h}(\xi)|^2\ud \xi\right)^{\frac12}\left(\int_{\left[-\frac{\pi}{h},\frac{\pi}{h}\right]^4}\left(|\xi|^4+\frac{1}{(d(x)+h)^4}\right)^{-1}\ud \xi\right)^{\frac12}\,.
\end{align*}
Using \eqref{eq:poinsob4} and \eqref{eq:poinsob5} we see that
\[\int_{\left[-\frac{\pi}{h},\frac{\pi}{h}\right]^4}\left(|\xi|^4+\frac{1}{(d(x)+h)^4}\right)|\widehat{v_h}(\xi)|^2\ud \xi\le C\|v_h\|_{W^{2,2}_h((h\Z)^4)}^2\,.\]
Furthermore we can compute using polar coordinates that
\begin{align*}
	\int_{\left[-\frac{\pi}{h},\frac{\pi}{h}\right]^4}\left(|\xi|^4+\frac{1}{(d(x)+h)^4}\right)^{-1}\ud \xi&=\int_{\left[-\frac{\pi}{h},\frac{\pi}{h}\right]^4}\frac{(d(x)+h)^4}{1+(d(x)+h)^4|\xi|^4}\ud \xi\\
	&\le C\int_0^{\frac{2\pi}{h}}\frac{(d(x)+h)^4s^3}{1+(d(x)+h)^4s^4}\ud s\\
	&\le C\log\left(1+(d(x)+h)^4\left(\frac{2\pi}{h}\right)^4\right)\\
	&\le C\log\left(2+\frac{d(x)}{h}\right)\,.
\end{align*}
Putting everything together we indeed arrive at
\begin{align*}
|u_h(x)|=|v_h(x)|\le C\sqrt{\log\left(2+\frac{d(x)}{h}\right)}\|v_h\|_{W^{2,2}_h((h\Z)^4)}\le C\sqrt{\log\left(2+\frac{d(x)}{h}\right)}\|u_h\|_{W^{2,2}_h((h\Z)^4)}\,.
\end{align*}
\end{proof}

\subsection{Estimates for finite difference schemes}\label{s:finitediff}

Let us discuss next the estimate for the approximation order of finite difference schemes that was already mentioned in the introduction.

To state it we need some definitions, taken from \cite{Muller2019b}. For $j\ge1$ let $\theta_j$ be the standard univariate centred B-spline of degree $j-1$ (cf. \cite[Section 1.9.4]{Jovanovic2014}). Of interest to us are
\begin{align*}
\theta_3(z):&=\begin{cases}\frac34-z^2&|z|\le\frac12\\
\frac12\left(|z|-\frac32\right)^2&\frac12<|z|\le\frac32\\
0&\text{else}\end{cases}\,,\\
\theta_1(z):&=\begin{cases}1&|z|\le\frac12\\
0&\text{else}\end{cases}\,.
\end{align*}
Using this, we can define the smoothing operator $T^{h,j}_i$ for $1\le i\le 4$ as
\[T^{h,j}_if(x):=\frac1h \int_\R f(x_1,\ldots,x_{i-1},y_i,x_{i+1},\ldots,x_4)\theta_j\left(\frac{x_i-y_i}{h}\right)\ud y_i\]
extended to distributions on $\R^4$ in the obvious way. Furthermore, we set
\[T^{h,j,\ldots, j}f:=T^{h,j}_1\circ\cdots \circ T^{h,j}_4f\,.\]
It is important for us that $T^{h,j}_i$ maps constant functions to themselves and that 
\[T^{h,j}_i\partial_i^2f=D^h_iD^h_{-i}T^{h,j-2}_if\,.\]
If we define the shorthand
\[T^{h,3,3,3,3-2e_i}:=T^{h,3}_1\circ\ldots \circ T^{h,3}_{i-1}\circ T^{h,1}_i\circ T^{h,3}_{i+1}\circ\ldots  \circ T^{h,3}_4\]
we also have
\begin{equation}\label{eq:tandfindiff}
T^{h,3,3,3,3}\partial_i^2f=D^h_iD^h_{-i}T^{h,3,3,3,3-2e_i}f\,.
\end{equation}
\begin{theorem}\label{t:mss}
Let $2<s<\frac52$, let $u\in W^{s,2}_0((0,1)^4)$, extended by 0 to $\tilde u\in W^{s,2}(\R^4)$.
Let $\Delta^2\tilde u=f$ as distributions, so that in particular 
\[\Delta^2u=f \qquad\text{in }(0,1)^4\,.\]
Furthermore, let $u_h\colon(h\Z)^4\to\R$ be the solution of
\begin{alignat*}{2}
		\Delta_h^2u_h&=T^{h,3,3,3,3}f\qquad && \text{in }V_h\\
		u_h& =0 \qquad &&\text{on }(h\Z)^4\setminus V_h\,.
\end{alignat*}
Then we have
\[\|u_h-\tilde u\|_{W^{2,2}_h((h\Z)^4)}\le C_sh^{s-2}\|u\|_{W^{s,2}((0,1)^4)}\,.\]
\end{theorem}
Note that $f=\Delta^2\tilde u\in W^{s-4,2}(\R^4)$ is in a negative Sobolev space. The operator $T^{h,3,3,3,3}$ maps $W^{t,2}(\R^4)$ to $C(\R^4)$ for any $t>-\frac52$ (see \cite[Section 1.9.4]{Jovanovic2014}). So in particular $T^{h,3,3,3,3}f$ has pointwise values and the difference scheme in Theorem \ref{t:mss} makes sense.

This theorem is closely related to \cite[Theorem 1.2]{Muller2019b}. In that theorem one takes $\frac52<s\le3$, and $T^{h,3,3,3,3}$ is replaced by $T^{h,2,2,2,2}$. The novelty of that work lies in choosing a good extension $\tilde u$ and dealing with its boundary values. In our case we can just extend $u$ by 0 and thereby avoid many of these subtleties. In fact, all the ideas for the proof of Theorem \ref{t:mss} are already for example in \cite{Jovanovic2014}.

To make this work more self-contained we give some details for a proof of Theorem \ref{t:mss}, closely following \cite{Muller2019b}.

\begin{proof}[Proof of Theorem \ref{t:mss}]
First of all, $s<\frac52$ and $u\in W^{s,2}_0((0,1)^4)$ imply that $\tilde u$ is actually in $W^{s,2}(\R^4)$ and $\|\tilde u\|_{W^{s,2}(\R^4)}=\|u\|_{W^{s,2}((0,1)^4)}$.

Let $e_h\colon(h\Z)^4\to\R$ be given by $e_h=\tilde u-u_h$. Then,
\begin{alignat*}{2}
		\Delta_h^2e_h&=\Delta_h^2\tilde u-\Delta_h^2u_h=\Delta_h^2\tilde u-T^{h,3,3,3,3}\Delta^2\tilde u \qquad&& \text{on }V_h\\
		e_h&=0 \qquad&& \text{on }(h\Z)^4\setminus V_h
\end{alignat*}
and by summation by parts we have
\begin{align}
\|\nabla_h^2e_h\|^2_{L^2_h((h\Z)^4)}=(e_h,\Delta_h^2e_h)_{L^2_h((h\Z)^4)}=(e_h,\Delta_h^2\tilde u-T^{h,3,3,3,3}\Delta^2\tilde u)_{L^2_h((h\Z)^4)}\,.\label{eq:identitye}
\end{align}
We can rewrite $\Delta_h^2\tilde u-T^{h,3,3,3,3}\Delta^2\tilde u$ using \eqref{eq:tandfindiff} as
\begin{align*}
	 \Delta_h^2\tilde u-T^{h,3,3,3,3}\Delta^2\tilde u&=\sum_{i=1}^4D^h_iD^h_{-i}\Delta_h\tilde u-T^{h,3,3,3,3}\partial_i^2\Delta\tilde{u}\\
	&=\sum_{i=1}^4D^h_iD^h_{-i}\Delta_h\tilde u-D^h_iD^h_{-i}T^{h,3,3,3,3-2e_i} \Delta\tilde{u}\\
	&=\sum_{i=1}^4D^h_iD^h_{-i}g_i
\end{align*}
where
\[g_i:=\Delta_h\tilde u-T^{h,3,3,3,3-2e_i}\Delta\tilde{u}\,.\]
We can insert this into \eqref{eq:identitye} and use summation-by-parts once again to obtain

\begin{align*}
\|\nabla_h^2e_h\|^2_{L^2_h((h\Z)^4)}&=\sum_{i=1}^4(e_h,D^h_iD^h_{-i}g_i)_{L^2_h((h\Z)^4)}\\
&= \sum_{i=1}^4(D^h_iD^h_{-i}e_h,g_i)_{L^2_h((h\Z)^4)}\\
&\le \sum_{i=1}^4\|g_i\|_{L^2_h((h\Z)^4)}\|\nabla_h^2e_h\|_{L^2_h((h\Z)^4)}
\end{align*}
and thus
\begin{equation}
	\|\nabla_h^2e_h\|_{L^2_h((h\Z)^4)}\le\sum_{i=1}^4\|g_i\|_{L^2_h((h\Z)^4)}\,.\label{eq:decomp_E}
\end{equation}

The summands on the right hand side can be bounded using the Bramble--Hilbert lemma (see e.g. \cite[Theorem 2.28]{Jovanovic2014}):
As $s>2$, 
\[|\Delta_h\tilde u(x)|\le C_h\|\tilde u\|_{L^\infty(x+(-3h/2,3h/2)^4)}\le C_{h,s}\|\tilde u\|_{H^s(x+(-3h/2,3h/2)^4)}\,.\]
Because $s>\frac 32$ and $T^{h,3,3,3,3-2e_i}f(x)$ only depends on $f|_{x+(-3h/2,3h/2)^4}$ we can conclude from \cite[Theorem 1.67]{Jovanovic2014} and the locality of $T^{h,3,3,3,3-2e_i}$ that
\[|T^{h,3,3,3,3-2e_i} \Delta\tilde{u}(x)|\le C_{h,s}\|\tilde u\|_{H^s(x+(-3h/2,3h/2)^4)}\,.\]
Thus $g_i(x)$ is a bounded linear functional of $\tilde u\in W^{s,2}(x+(-3h/2,3h/2)^4)$. This functional vanishes when $\tilde u|_{x+(-3h/2,3h/2)^4}$ is a polynomial of degree at most 2. Indeed, if that is the case then $\Delta\tilde u|_{x+(-3h/2,3h/2)^4}$ is a constant function, and $\Delta_h\tilde u(x)$ is equal to the same constant, and the claim follows from the fact that $T^{h,3}_1\ldots T^{h,3}_{i-1}T^{h,1}_i T^{h,3}_{i+1}\ldots  T^{h,3}_4$  maps constant functions to themselves.

We have now shown that $g_i(x)$ is a bounded linear functional of $\tilde u\in W^{s,2}(x+(-3h/2,3h/2)^4)$ that vanishes on polynomials of degree at most 2.
By the Bramble--Hilbert lemma it is bounded by $C_{h,s}[\tilde u]_{W^{s,2}(x+(-3h/2,3h/2)^4)}$ for $s\le3$.
Using a scaling argument to determine the correct prefactor of $h$, we obtain
\[|g_i(x)|\le C_sh^{s-4}[\tilde u]_{W^{s,2}(x+(-3h/2,3h/2)^4)}\]
and hence
\begin{align}
	\|g_i\|_{L^2_h((h\Z)^4)}^2&\le Ch^4\sum_{x\in(h\Z)^4}h^{2(s-4)}[\tilde u]_{W^{s,2}(x+(-3h/2,3h/2)^4)}^2\nonumber\\
	&\le C_sh^{2(s-2)}[\tilde u]_{W^{s,2}(\R^4)}^2 \le C_sh^{2(s-2)}\|u\|_{W^{s,2}((0,1)^4)}^2\label{eq:est_phii}
\end{align}
for those $s$.
Now we can plug \eqref{eq:est_phii} into \eqref{eq:decomp_E} and obtain
\[\|\nabla_h^2e_h\|_{L^2_h((h\Z)^4)}\le C_sh^{s-2}\|u\|_{W^{s,2}((0,1)^4)}\]
for $s<\frac52$. Using the discrete Poincar\'e inequality completes the proof.
\end{proof}

\section{Estimates for other Green's functions}\label{s:estgreen}
\subsection{Estimates for the discrete Green's function of the full space}\label{s:estgreendisc}
Our strategy will be to compare $G_h$ with several other Green's functions, so let us introduce these first.

Recall that $\lambda=\sqrt{8}\pi$. Let $G$ be the Green's function of the continuous Bilaplacian on $[0,1]^4$ with Dirichlet boundary data (i.e. of the problem \eqref{eq:bilaplace_cont}). We also need Green's functions on the full space. Let $\hat G(x,y):=-\frac{1}{\lambda^2}\log|x-y|$. It is easy to check that this is a fundamental solution of the Bilaplacian (i.e. that $\Delta^2\left(-\frac{1}{\lambda^2}\log|\cdot-y|\right)=\delta_y$ in the sense of distributions). We also define $ \hat G_h\colon(h\Z)^4\times(h\Z)^4\to\R$ by $\hat G_h(x,y)=F\left(\frac{x-y}{h}\right)-\frac{1}{\lambda^2}\log h$ where $F$ is the function introduced in the following lemma. We added the summand $-\frac{1}{\lambda^2}\log h$ here to ensure that $\hat G_h$ has the same asymptotic behaviour as $\hat G$. We also define shifted versions of $\hat G_h$ and $\hat G$, namely for $r>0$ we let $\hat G^{(r)}=\hat G+\frac{\log r}{\lambda^2}$, and $\hat G_h^{(r)}=\hat G_h+\frac{\log r}{\lambda^2}$. We occasionally write $G_y$ for $G(\cdot,y)$, and define $G_{h,y}$, $\hat G_y$, $\hat G_{h,y}$, $\hat G_y^{(r)}$ and $\hat G_{h,y}^{(r)}$ analogously.

\begin{lemma}[{\cite[pp. 96-97]{Mangad1967}}]\label{l:mangad}
There is a function $F\colon\Z^4\to\R$ such that $\Delta_1^2F(x)=\begin{cases}1&x=0\\0&\text{else} \end{cases}$, satisfying the asymptotics
\[F(x)=-\frac{1}{8\pi^2}\log|x|+\frac{1}{24\pi^2}\frac{x_1^4+x_2^4+x_3^4+x_4^4}{|x|^6}+O\left(\frac{1}{|x|^4}\right)\]
for $x\ne0$.
\end{lemma}
In \cite{Mangad1967}, $F$ is defined using the discrete Fourier multiplier associated to $\Delta_1^2$. By expanding that multiplier into a Laurent series and computing the Fourier transform termwise it is possible to give asymptotic expansions to arbitrary high order. This technique also applies to other discrete polyharmonic Green's functions. For our purposes the first two terms quoted above are sufficient.

Lemma \ref{l:mangad} immediately gives us an asymptotic expansion of $\hat G_h$, and so we can easily obtain estimates for $\hat G_h$ and $\hat G_h^{(r)}$.
\begin{lemma}\label{l:estgreenfull}
Let $h>0$, and $r\ge 192h$. Let $\alpha\in\N^4$ with $|\alpha|\le2$. Then for any $x,y\in(h\Z)^4$ with $\frac{r}{64}\le|x-y|_\infty\le 16r$ we have
\begin{align}
	\left|\hat G_h^{(r)}(x,y)-\frac{1}{\lambda^2}\log\left(\frac{r}{|x-y|+h}\right)\right|&\le C\label{eq:estgreenfull3}\,,\\
	\left|D^h_\alpha \hat G_{h,y}^{(r)}(x)\right|&\le \frac{C}{r^{|\alpha|}}\label{eq:estgreenfull1}\,,\\
	\left|D^h_\alpha \hat G_{h,y}^{(r)}(x)-\partial^\alpha \hat G_y^{(r)}(x)\right|&\le C\frac{h}{r^{|\alpha|+1}}\,.\label{eq:estgreenfull2}
\end{align} 
\end{lemma}
\begin{proof}
By translation invariance we may assume $y=0$. The definition of $\hat G_h^{(r)}$ implies that
\begin{align}
\hat G_h^{(r)}(x,0)&=F\left(\frac{x}{h}\right)-\frac{1}{\lambda^2}\log h+\frac{1}{\lambda^2}\log r\nonumber\\
&=-\frac{1}{\lambda^2}\log\frac{|x|}{h} +\frac{h^2}{24\pi^2}\frac{x_1^4+x_2^4+x_3^4+x_4^4}{|x|^6}+O\left(\frac{h^4}{|x|^4}\right)-\frac{1}{\lambda^2}\log h+\frac{1}{\lambda^2}\log r\nonumber\\
&=\frac{1}{\lambda^2}\log\frac{r}{|x|}+\frac{h^2}{24\pi^2}\frac{x_1^4+x_2^4+x_3^4+x_4^4}{|x|^6}+O\left(\frac{h^4}{|x|^4}\right)\,. \label{eq:estgreenfull4} 
\end{align}
From this we immediately conclude \eqref{eq:estgreenfull3} in the case $x\neq0$. In case $x=0$ we can directly use
\[\hat G_h^{(r)}(0,0)=F(0)+\frac{1}{\lambda^2}\log \frac{r}{h}\]
to obtain \eqref{eq:estgreenfull3}.

The explicit formula for $\hat G$ reveals that
\[\left|\partial^\alpha \hat G_0^{(r)}(x)\right|=\left|\partial^\alpha\frac{1}{\lambda^2}\log\frac{r}{|x|}\right|\le\frac{C}{r^{|\alpha|}} \]
if $\frac r{64}\le|x|_\infty$, and thus \eqref{eq:estgreenfull1} easily follows from \eqref{eq:estgreenfull2}.

For \eqref{eq:estgreenfull2} we want to take discrete derivatives of each summand in \eqref{eq:estgreenfull4} separately. If $g=O\left(\frac{h^4}{|\cdot|^4}\right)$ then $|D^h_\alpha g(x)|\le\frac{C}{h^{|\alpha|}}\frac{h^4}{|x|^4}=C\frac{h^{4-|\alpha|}}{|x|^4}$
so for $|\alpha|\le2$ we can neglect the error term. Using Taylor's theorem we can see that
\[D^h_\alpha\left(\frac{1}{\lambda^2}\log\frac{r}{|x|}+\frac{h^2}{24\pi^2}\frac{x_1^4+x_2^4+x_3^4+x_4^4}{|x|^6}\right)=\partial^\alpha\frac{1}{\lambda^2}\log\frac{r}{|x|}+O\left(\frac{h}{|x|^{|\alpha|+1}}\right)\,.\]
Note that we can avoid the singularity here because $|x|\ge\frac r{64}\ge3h$. This easily implies \eqref{eq:estgreenfull2}.
\end{proof}

\subsection{Estimates for continuous Green's functions}\label{s:estgreencont}
We want to compare $G$ and $G_h$. This is only useful if we also have estimates for $G$ to begin with. We will derive such estimates in this section. The following estimates are far from optimal, but sufficient for our purposes.

We obviously have a well-posedness result for the Bilaplace equation in the energy space $W^{2,2}$. The following result states that the same holds true if we raise the regularity slightly.
\begin{theorem}\label{t:wellposednessWs2}
There exists $\kappa_0>0$ with the following property: Let $0\le\kappa\le\kappa_0$. Then for each $f\in W^{-2+\kappa,2}((0,1)^4)$ there is a unique $u\in W^{2+\kappa,2}\cap W^{2,2}_0((0,1)^4)$ such that $\Delta^2u=f$ in the sense of distributions, and we have the estimate
\begin{equation}\label{e:wellposednessWs2_e1}
	\|u\|_{W^{2+\kappa,2}((0,1)^4)}\le C_\kappa\|f\|_{W^{-2+\kappa,2}((0,1)^4)}
\end{equation}
for a constant $C_\kappa$ depending only on $\kappa$.
\end{theorem}
For convenience we will assume in the following that $\kappa_0<\frac12$, and fix such a $\kappa_0$. Note that $W^{2+\kappa,2}\cap W^{2,2}_0((0,1)^4)=W^{2+\kappa,2}_0((0,1)^4)$ if $\kappa<\frac12$.
\begin{proof}[Proof of Theorem \ref{t:wellposednessWs2}]
This is a special case e.g. of \cite[Theorem 6.32]{Mitrea2013}, but for the convenience of the reader we give the short argument.

We begin with the case $\kappa=0$. In that case we can test the weak form of $\Delta^2u=f$ with $u$ and obtain
\begin{align*}
\|\nabla^2u\|_{L^2((0,1)^4)}^2=(u,\Delta^2u)_{L^2((0,1)^4)}=(u,f)_{L^2((0,1)^4)}\le\|u\|_{W^{2,2}((0,1)^4)}\|f\|_{W^{-2,2}((0,1)^4)}\,.
\end{align*}
The Poincar\'e inequality implies $\|u\|_{W^{2,2}((0,1)^4)}\le C\|\nabla^2u\|_{L^2((0,1)^4)}$ and so we obtain \eqref{e:wellposednessWs2_e1}.

For the general case we can use a stability result for analytic families of operators on Banach spaces: The spaces $W^{s,2}((0,1)^4)$ and $W_0^{s,2}((0,1)^4)$ each form an interpolation family with respect to complex interpolation, and so by \cite[Proposition 4.1]{Tabacco1988} the set of those $s$ for which $\Delta^2\colon W_0^{s,2}((0,1)^4)\to W^{s-4,2}((0,1)^4)$ has a bounded inverse is open. We know that this set contains $2$, so the existence of $\kappa_0$ as in the theorem follows.
\end{proof}

Next we state some estimates for $G$. We begin by estimating the regular part of $G$ in certain Sobolev norms. Recall that $\hat G^{(r)}(x,y)=\hat G(x,y)+\frac{\log r}{\lambda^2}$ for any $r>0$.

\begin{lemma}\label{l:estimateGWs2}
Let $\kappa_0$ be as in Theorem \ref{t:wellposednessWs2}, and let $0\le\kappa\le\kappa_0$. Let $K\ge2$, $r>0$, $y\in(0,1)^4$ be such that $\frac{d(y)}{K}\le r\le\frac{d(y)}{2}$. Then
\begin{equation}
	\left\|G_y-\eta_y^{(r)}\hat G_y^{(r)}\right\|_{W^{2+\kappa,2}((0,1)^4)}\le\frac{C_{K,\kappa}}{r^\kappa}
\end{equation}
for a constant $C_{K,\kappa}$ depending only on $K$ and $\kappa$.
\end{lemma}
\begin{proof}
Let $H^{(r)}=G_y-\eta_y^{(r)}\hat G_y^{(r)}$. By Theorem \ref{t:wellposednessWs2} it suffices to show
\begin{equation}
	\|\Delta^2H^{(r)}\|_{W^{-2+\kappa,2}((0,1)^4)}\le\frac{C_{K,\kappa}}{r^\kappa}\,.\label{e:estimateGWs2}
\end{equation}
By standard interpolation theory and our assumption $\kappa\in[0,\kappa_0]\subset[0,2]$ it suffices to establish this for $\kappa\in\{0,2\}$.

Observe that $\Delta^2H^{(r)}$ is zero in $(0,1)^4\setminus B_r(y)$ as well as in $B_{r/2}(y)$ (as the two singularities cancel out). This means that $\Delta^2 H^{(r)}$ is supported in $B_r(y)\setminus B_{r/2}(y)$ and there it is equal to $-\Delta^2\left(\eta_y^{(r)}\hat G_y^{(r)}\right)$. We have an explicit formula for $\hat G_y^{(r)}$, and so it is straightforward to check that
$\left|\Delta^2\left(\eta_y^{(r)}\hat G_y^{(r)}\right)\right|$ is bounded by $\frac{C_K}{r^4}$ on $B_r(y)\setminus B_{r/2}(y)$. This easily implies \eqref{e:estimateGWs2} for $\kappa=2$.

For the case $\kappa=0$ we need to be slightly more careful:
Let $\chi_y^{(r)}$ be a cut-off function that is 1 on $B_r(y)\setminus B_{r/2}(y)$ and zero outside $B_{2r}(y)\setminus B_{r/4}(y)$ (e.g. $\chi_y^{(r)}=\eta_y^{(2r)}-\eta_y^{(r/2)}$). Then we have $\Delta^2H^{(r)}=-\chi_y^{(r)}\Delta^2\left(\eta_y^{(r)}\hat G_y^{(r)}\right)$ and thus we can calculate
\begin{align*}
	\|\Delta^2H^{(r)}\|_{W^{-2,2}((0,1)^4)}&=\sup_{\|\varphi\|_{ W^{2,2}_0((0,1)^4)}=1}\int\Delta^2H^{(r)}\varphi\\
	&=\sup_{\|\varphi\|_{ W^{2,2}_0((0,1)^4)}=1}\int-\Delta^2\left(\eta_y^{(r)}\hat G_y^{(r)}\right)\chi_y^{(r)}\varphi\\
	&=\sup_{\|\varphi\|_{ W^{2,2}_0((0,1)^4)}=1}\int-\Delta\left(\eta_y^{(r)}\hat G_y^{(r)}\right)\Delta(\chi_y^{(r)}\varphi)\\
	&\le C\left\|\Delta\left(\eta_y^{(r)}\hat G_y^{(r)}\right)\right\|_{L^2(B_{2r}(y)\setminus B_{r/4}(y))}\sup_{\|\varphi\|_{ W^{2,2}_0((0,1)^4)}=1}\|\Delta(\chi_y^{(r)}\varphi)\|_{L^2((0,1)^4)}\,.
\end{align*}
To estimate the second factor we proceed as in the calculation that led to \eqref{eq:poinsob6}. We have a Poincar\'e inequality
\begin{equation}
\|u\|_{L^2(z+(-s,s)^4)}\le Cs\|\nabla u\|_{L^2(z+(-s,s)^4)}\label{e:poincarezeroonface}
\end{equation}
for any $u\in W^{1,2}(z+(-s,s)^4)$ that is zero (in the sense of traces) on one of the faces of $z+(-s,s)^4$. This is the continuous analogue to Lemma \ref{l:poincare}, and the proof is very similar. Using \eqref{e:poincarezeroonface} we can estimate
\begin{align*}
	\|\Delta(\chi_y^{(r)}\varphi)\|_{L^2((0,1)^4)}&\le C\|\nabla^2\varphi\|_{L^2(B_{d(y)}(y))}+\frac{C}{r}\|\nabla\varphi\|_{L^2(B_{d(y)}(y))}+\frac{C}{r^2}\|\varphi\|_{L^2(B_{d(y)}(y))}\\
	&\le C\left(1+\frac{d(y)}{r}+\frac{d(y)^2}{r^2}\right)\|\nabla^2\varphi\|_{L^2(y+(-d(y),d(y))^4)}\\
	&\le C_K\|\varphi\|_{W^{2,2}_0((0,1)^4)}\,.
\end{align*}

We also have that $\Delta\left(\eta_y^{(r)}\hat G^{(r)}_y\right)$ is bounded by $\frac{C}{r^2}$ on $B_{2r}(y)\setminus B_{r/4}(y)$ and hence 
\[\left\|\Delta\left(\eta_y^{(r)}\hat G_y^{(r)}\right)\right\|_{L^2(B_{2r}(y)\setminus B_{r/4}(y))}\le Cr^2\cdot\frac{1}{r^2}= C\,.\]
Using this we obtain \eqref{e:estimateGWs2} for $\kappa=0$.
\end{proof}

Next we give some estimates on the local behaviour of $G$. The first two allow us to control $G$ far from and close to the singularity, respectively, while the last one expresses the H\"older continuity of $G-\hat G$ near the diagonal.
\begin{lemma}\label{l:estimateGpointwise}
Let $\kappa_0$ be as in Theorem \ref{t:wellposednessWs2}. Let $y\in(0,1)^4$. The function $G_y$ is smooth on $(0,1)^4\setminus\{y\}$, and $G-\hat G$ is symmetric and smooth on $(0,1)^4\times(0,1)^4\setminus\{(x,x)\colon x\in(0,1)^4\}$ and can be extended continuously to $(0,1)^4\times(0,1)^4$. Slightly abusing notation, we write
\[G(y,y)-\hat G(y,y):=\lim_{\substack{(y',y'')\to(y,y)\\y'\neq y''}}G(y',y'')-\hat G(y',y'')\,.\]
Let $K\ge1$. We have the following estimates, where $\frac{d(y)}{K}\le r\le \frac{d(y)}{2}$:
\begin{align}
|G(x,y)|\le C\quad&\text{if }|x-y|\ge \frac{d(y)}{4}\,,\label{e:boundGfar}\\
|G(x,y)-\hat G^{(r)}(x,y)|\le C_K\quad&\text{if }|x-y|\le d(y)\,.\label{e:boundGnear}
\end{align}
Furthermore if $r>0$ is arbitrary, $|y'-y|\le \frac{d(y)}{8}$ and $|y''-y|\le \frac{d(y)}{8}$ we have the estimate
\begin{align}
\left|G(y',y'')-\hat G^{(r)}(y',y'')-\left(G(y,y)-\hat G^{(r)}(y,y)\right)\right|\le C\frac{|y'-y|^{\kappa_0}+|y''-y|^{\kappa_0}}{d(y)^{\kappa_0}}\,.\label{e:boundGholder}
\end{align}
\end{lemma}
\begin{proof}
The smoothness of $G$ and $G-\hat G$ follows from standard regularity theory for higher order elliptic equations. The estimate \eqref{e:boundGfar} is given in \cite[Theorem 8.1]{Mayboroda2014}. There also a variant of \eqref{e:boundGnear} (without the correction $\frac{\log r}{\lambda^2}$ and with slightly worse error term) is given. The results in \cite{Mayboroda2014} however are in a far more general setting, so we prefer to give an elementary proof of the specific estimates we need.

We use a standard Caccioppoli inequality (see e.g. \cite[Capitolo II, Teorema 3.II or Teorema 6.I]{Campanato1980}): If $u\in W^{2,2}(B_s(z))$ and $\Delta^2u=0$ in $B_s(z)$ then
\begin{equation}\label{e:caccioppoli}
\|\nabla^2u\|_{L^\infty(B_{s/2}(z))}\le \frac{C}{s^2}\|\nabla^2u\|_{L^2(B_s(z))}\,.
\end{equation}
We will also need a special case of the Gagliardo-Nirenberg interpolation inequality, namely
\begin{equation}\label{e:estimateGpointwise_2}
\|u\|_{L^\infty(B_s(z))}\le C\left(s^2\|\nabla^2u\|_{L^\infty(B_s(z))}+\frac{1}{s^2}\|u\|_{L^2(B_s(z))}\right)\,.
\end{equation}
To see this, observe first that by scaling we can assume $s=1$. The Poincar\'e inequality implies that $\|u-a-b\cdot (\cdot-z)\|_{L^\infty(B_1(z))}\le C\|\nabla^2u\|_{L^\infty(B_1(z))}$, where $a=\frac{1}{|B_1|}\int u$ and $b=\frac{1}{|B_1|}\int \nabla u$, and so we only have to bound $a$ and $b$. We have $|a|\le C\|u\|_{L^2(B_1(z))}$, and the estimate $\|u-a\|_{L^2(B_1(z))}\le\|u\|_{L^2(B_1(z))}$ implies 
\begin{align*}
|b|&\le C\|b\cdot (\cdot-z)\|_{L^2(B_1(z))}\\
&\le C\left(\|u-a-b\cdot (\cdot-z)\|_{L^2(B_1(z))}+\|u-a\|_{L^2(B_1(z))}\right)\\
&\le C\left(\|\nabla^2u\|_{L^\infty(B_1(z))}+\|u\|_{L^2(B_1(z))}\right)\,.
\end{align*}
This completes the proof of \eqref{e:estimateGpointwise_2}.

After these preparations we can now begin with the proof of \eqref{e:boundGfar}. We first assume that $d(x)\le 2d(y)$. Let $H^{(d(y)/8)}=G_y-\eta_y^{(d(y)/8)}\hat G_y^{(d(y)/8)}$. Lemma \ref{l:estimateGWs2} with $\kappa=0$ implies that 
\begin{equation}\label{e:estimateGpointwise_1}
	\|\nabla^2H^{(d(y)/8)}\|_{L^2((0,1)^4)}\le C\,.
\end{equation}
The function $H^{(d(y)/8)}$ agrees with $G_y$ on $(0,1)^4\setminus B_{d(y)/8}(y)$. Because $\frac{d(x)}{16}+\frac{d(y)}{8}\le \frac{d(y)}{4}\le |x-y|_\infty$ we have $B_{d(x)/16}(x)\cap B_{d(y)/8}(y)=\varnothing$ and thus \eqref{e:estimateGpointwise_1} implies
\[\|\nabla^2G_y\|_{L^2(B_{d(x)/16}(x))}\le C\,.\]
Using the Caccioppoli inequality \eqref{e:caccioppoli} we conclude
\begin{equation}
\|\nabla^2G_y\|_{L^\infty(B_{d(x)/32}(x))}\le \frac{C}{d(x)^2}\,.\label{e:estimateGpointwise_3}
\end{equation}
Next, note that the Poincar\'e inequality \eqref{e:poincarezeroonface} applied on $x+(-d(x),d(x))^4$ and \eqref{e:estimateGpointwise_1} imply that
\[\|H^{(d(y)/8)}\|_{L^2(B_{d(x)}(x))}\le Cd(x)^2\|\nabla^2H^{(d(y)/8)}\|_{L^2(x+(-d(x),d(x))^4)}\le Cd(x)^2\]
and therefore
\[\|G_y\|_{L^2(B_{d(x)/32}(x))}\le Cd(x)^2\,.\]
Recalling \eqref{e:estimateGpointwise_3}, an application of \eqref{e:estimateGpointwise_2} concludes the proof.

It remains to consider the case $d(x)>2d(y)$. In that case $|x-y|\ge d(x)-d(y)\ge\frac{d(x)}{2}$, so we can interchange the roles of $x$ and $y$ and repeat the above proof (using that $G(x,y)=G(y,x)$).

Next we give a proof of \eqref{e:boundGnear}. This is quite similar to the preceding argument. Because $G^{(r)}$ differs from $G^{(d(y))}$ only by at most $\frac{1}{\lambda^2}\log K\le C_K$ we can assume $r=d(y)$. Let again $H^{(d(y))}=G_y-\eta_y^{(d(y))}\hat G_y^{(d(y))}$.
Observe first that if $|x-y|\ge\frac{d(y)}{4}$ then \eqref{e:boundGfar} implies \eqref{e:boundGnear}. Therefore we can restrict our attention to the case $|x-y|\le\frac{d(y)}{4}$. By Lemma \ref{l:estimateGWs2} we have that
\[\|\nabla^2H^{(d(y))}\|_{L^2((0,1)^4)}\le C\,.\]
The function $H^{(d(y))}$ agrees with $G_y-\hat G_y^{(d(y))}$ on $B_{d(y)/2}(y)$. Thus, as before, the Caccioppoli inequality implies that
\[\|\nabla^2(G_y-\hat G_y)\|_{L^\infty(B_{d(y)/4}(y))}\le \frac{C}{d(y)^2}\]
and the Poincar\'e inequality implies
\[\|G_y-\hat G_y\|_{L^2(B_{d(y)/4}(y))}\le \|H^{(d(y))}\|_{L^2(B_{d(y)}(y))}\le Cd(y)^2\]
so that the conclusion follows from the interpolation inequality \eqref{e:estimateGpointwise_2}.

For \eqref{e:boundGholder} observe that by Lemma \ref{l:estimateGWs2} we control the $W^{2+\kappa_0,2}$-norm of $G_y-\eta_y^{(d(y))}\hat G_y^{(d(y))}$. That Sobolev space embeds into the H\"older space $C^{0,\kappa_0}$ and so we have
\begin{align*}
\left[G_y-\eta_y^{(d(y))}\hat G_y^{(d(y))}\right]_{C^{0,\kappa_0}((0,1)^4)}&\le C\left\|G_y-\eta_y^{(d(y))}\hat G_y^{(d(y))}\right\|_{W^{2+\kappa_0,2}((0,1)^4)}\le \frac{C}{d(y)^{\kappa_0}}\,.
\end{align*}
Because $G_y-\eta_y^{(d(y))}\hat G_y^{(d(y))}$ agrees with $G_y-\hat G_y^{(d(y))}$ on $B_{d(y)/2}(y)$ this implies 
\[\left|G(y',y)-\hat G^{(d(y))}(y',y)-\left(G(y,y)-\hat G^{(d(y))}(y,y)\right)\right|\le C\frac{|y'-y|^{\kappa_0}}{d(y)^{\kappa_0}}\,.\]
If we add and subtract $\frac{\log r-\log d(y)}{\lambda^2}$ on the left-hand side we obtain
\[\left|G(y',y)-\hat G^{(r)}(y',y)-\left(G(y,y)-\hat G^{(r)}(y,y)\right)\right|\le C\frac{|y'-y|^{\kappa_0}}{d(y)^{\kappa_0}}\,.\]
Similarly we obtain
\[\left|G(y'',y')-\hat G^{(r)}(y'',y')-\left(G(y,y')-\hat G^{(r)}(y,y')\right)\right|\le C\frac{|y''-y|^{\kappa_0}}{d(y')^{\kappa_0}}\]
where we used that $d(y')\ge\frac78d(y)$ so that $y,y''\in B_{d(y')/2}(y')$. If we add the last two estimates and use once again that $d(y')\ge\frac78d(y)$ we arrive at \eqref{e:boundGholder}.
\end{proof}

\section{Proof of the main theorems}\label{s:proofmainthm}
In this section we will finally prove that $G_h$ satisfies \ref{B.0'}, \ref{B.1'}, \ref{B.2'}, \ref{B.3'}, which according to Observation \ref{o:estimatesGh} implies Theorem \ref{t:membranemodifiedassumptions}.

Recall that $G_h$ is the Green's function of the discrete Bilaplacian on $V_h$ with zero boundary data outside $V_h$, $G$ is the Green's function of the continuous Bilaplacian on $(0,1)^4$ with zero Dirichlet boundary data, and $\hat G_h^{(r)}$ and $\hat G^{(r)}$ are shifted versions of the discrete and continuous full space Green's function.

The main technical statement used in the proof of Theorem \ref{t:membranemodifiedassumptions} will be the following.
\begin{lemma}\label{l:discandcontGreenareclose}
Let $\kappa_0$ be as in Theorem \ref{t:wellposednessWs2}. Let $K\ge2$, and $r\ge192h$. Then for all $x,y\in V_h$ with $\frac{d(y)}{K}\le r\le\frac{d(y)}{2}$ we have
\begin{align*}
&\left|\left(G_h(x,y)-\eta^{(r)}_{h,y}(x)\hat G^{(r)}_h(x,y)\right)- \left(G(x,y)-\eta_y^{(r)}(x)\hat G^{(r)}(x,y)\right)\right|\le C_K\frac{h^{{\kappa_0}}}{r^{\kappa_0}}\sqrt{\log\left(2+\frac{d(x)}{h}\right)}\,.
\end{align*}
\end{lemma}
This lemma is so useful because it simultaneously provides control over the difference between the discrete and continuous Green's function when $x,y$ are far apart and over the difference of the regular part of the discrete and continuous Green's function when $x,y$ are close. 
\begin{proof}[Proof of Lemma \ref{l:discandcontGreenareclose}]
We define $H_h=G_{h,y}-\eta^{(r)}_{h,y}\hat G^{(r)}_{h,y}$ and $H=G_y-\eta_y^{(r)}\hat G_y^{(r)}$.
Let $\tilde H_h$ be the solution of 
\[\begin{array}{rl}
  \Delta_h^2\tilde H_h=T^{h,3,3,3,3}\Delta^2H & \text{in }V_h\\
  \tilde H_h=0 & \text{on }(h\Z)^4\setminus V_h\,.
\end{array}\]
Our goal is to estimate $|H_h(x)-H(x)|$. We will estimate $H_h-\tilde H_h$ and $\tilde H_h-H$ separately.

The estimate of the latter term is straightforward: Using Theorem \ref{t:mss} and Lemma \ref{l:estimateGWs2}, we obtain 
\begin{align*}
	\|\tilde H_h-H\|_{W^{2,2}_h((h\Z)^4)}\le C_Kh^{\kappa_0}\|H\|_{W^{2+\kappa_0,2}((0,1)^4)}\le C_K\frac{h^{\kappa_0}}{r^{\kappa_0}}\,.
\end{align*}	

Estimating $H_h-\tilde H_h$ is more tedious. Similarly as in the proof of Lemma \ref{l:estimateGWs2} we let $\chi_y^{(r)}=\eta_y^{(4r)}-\eta_y^{(r/4)}$ and $\chi^{(r)}_{h,y}$ be the restriction of $\chi_y^{(r)}$ to $(h\Z)^4$.
Then we have
\begin{align}
	\Delta_h^2(H_h-\tilde H_h)&=\Delta_h^2\left(G_{h,y}-\eta_{h,y}^{(r)}\hat G_{h,y}^{(r)}\right)- T^{h,3,3,3,3}\Delta^2\left(G_y-\eta_y^{(r)}\hat G_y^{(r)}\right)\nonumber\\
	&=\chi^{(r)}_{h,y}\Delta_h^2\left(G_{h,y}-\eta_{h,y}^{(r)}\hat G_{h,y}^{(r)}\right)- \chi^{(r)}_{h,y}T^{h,3,3,3,3}\Delta^2\left(G_y-\eta_y^{(r)}\hat G_y^{(r)}\right)\nonumber\\
	&=-\chi^{(r)}_{h,y}\Delta_h^2\left(\eta_{h,y}^{(r)}\hat G_{h,y}^{(r)}\right)+ \chi^{(r)}_{h,y}T^{h,3,3,3,3}\Delta^2\left(\eta_y^{(r)}\hat G_y^{(r)}\right)\nonumber\\
	&=-\chi^{(r)}_{h,y}\Delta_h^2\left(\eta_{h,y}^{(r)}\hat G_{h,y}^{(r)}\right)+ \sum_{i=1}^4\chi^{(r)}_{h,y}T^{h,3,3,3,3}\partial_i^2\Delta\left(\eta_y^{(r)}\hat G_y^{(r)}\right)\nonumber\\
	&=-\chi^{(r)}_{h,y}\sum_{i=1}^4D^h_iD^h_{-i}\left(\Delta_h\left(\eta_{h,y}^{(r)}\hat G_{h,y}^{(r)}\right)+ T^{h,3,3,3,3-2e_i}\Delta\left(\eta_y^{(r)}\hat G_y^{(r)})\right)\right)\label{eq:discandcontGreenareclose5}\,.
\end{align}
Because $H_h-\tilde H_h$ is supported in $V_h$ we have
\begin{align*}
\|\nabla_h^2(H_h-\tilde H_h)\|^2_{L^2_h((h\Z)^4)}&=\left(\Delta_h^2(H_h-\tilde H_h),H_h-\tilde H_h\right)_{L^2_h((h\Z)^4)}\\
&\le \sup_{\substack{\varphi_h=0\text{ on }(h\Z)^4\setminus V_h\\\|\varphi_h\|_{ W^{2,2}_h((h\Z)^4)}=1}}\left(\Delta_h^2(H_h-\tilde H_h),\varphi_h\right)_{L^2_h((h\Z)^4))}\|H_h-\tilde H_h\|_{W^{2,2}_h((h\Z)^4)}
\end{align*}
which together with the Poincar\'e inequality implies that
\[\|H_h-\tilde H_h\|_{W^{2,2}_h((h\Z)^4)}\le C\sup_{\substack{\varphi_h=0\text{ on }(h\Z)^4\setminus V_h\\\|\varphi_h\|_{ W^{2,2}_h((h\Z)^4)}=1}}\left(\Delta_h^2(H_h-\tilde H_h),\varphi_h\right)_{L^2_h((h\Z)^4))}\,.\]
Combining this with \eqref{eq:discandcontGreenareclose5}, and abbreviating $T^*_i:=T^{h,3,3,3,3-2e_i}$ we see that
\begin{align}
&\|H_h-\tilde H_h\|_{W^{2,2}_h((h\Z)^4)}\nonumber\\
&\quad\le C\sup_{\substack{\varphi_h=0\text{ on }(h\Z)^4\setminus V_h\\
\|\varphi_h\|_{ W^{2,2}_h((h\Z)^4)}=1}}\sum_{i=1}^4\left(D^h_iD^h_{-i}\left(-\Delta_h\left(\eta_{h,y}^{(r)}\hat G_{h,y}^{(r)}\right)+ T^*_i\Delta\left(\eta_y^{(r)}\hat G_y^{(r)}\right)\right),\chi^{(r)}_{h,y}\varphi_h\right)_{L^2_h((h\Z)^4))}\nonumber\\
&\quad\le C\sup_{\substack{\varphi_h=0\text{ on }(h\Z)^4\setminus V_h\\
\|\varphi_h\|_{ W^{2,2}_h((h\Z)^4)}=1}}\sum_{i=1}^4\left(-\Delta_h\left(\eta_{h,y}^{(r)}\hat G_{h,y}^{(r)}\right)+ T^*_i\Delta\left(\eta_y^{(r)}\hat G_y^{(r)}\right),D^h_iD^h_{-i}\chi^{(r)}_{h,y}\varphi_h\right)_{L^2_h((h\Z)^4))}\nonumber\\
&\quad\le C\sum_{i=1}^4\left\|-\Delta_h\left(\eta_{h,y}^{(r)}\hat G_{h,y}^{(r)}\right)+ T^*_i\Delta\left(\eta_y^{(r)}\hat G_y^{(r)}\right)\right\|_{L^2_h(Q^h_{8r}(y)\setminus Q^h_{r/32}(y))}\nonumber\\
&\quad\qquad\times\sup_{\substack{\varphi_h=0\text{ on }(h\Z)^4\setminus V_h\\\|\varphi_h\|_{ W^{2,2}_h((h\Z)^4)}=1}}\left\|\nabla^2_h\left(\chi^{(r)}_{h,y}\varphi_h\right)\right\|_{L^2_h((h\Z)^4))}\,,\label{eq:discandcontGreenareclose1}
\end{align}
where we used that $\chi^{(r)}_{h,y}$ is supported in $B_{4r}(y)\setminus B_{r/8}(y)$ so that the support of $\Delta_h\left(\chi^{(r)}_{h,y}\varphi_h\right)$ is certainly contained in $Q^h_{8r}(y)\setminus Q^h_{r/32}(y)$.
The discrete product rule and the Poincar\'e inequality imply that
\begin{align*}
	&\left\|\nabla_h^2\left(\chi^{(r)}_{h,y}\varphi_h\right)\right\|_{L^2_h((h\Z)^4))}\\
	&\quad\le C\|\nabla_h^2\varphi_h\|_{L^2_h(Q^h_{d(y)+h}(y))}+\frac{C}{r}\|\nabla_h\varphi_h\|_{L^2_h(Q^h_{d(y)+h}(y))}+\frac{C}{r^2}\|\varphi_h\|_{L^2_h(Q^h_{d(y)+h}(y))}\\
	&\quad\le C\left(1+\frac{d(y)+h}{r}+\frac{(d(y)+h)^2}{r^2}\right)\|\varphi_h\|_{ W^{2,2}_h((h\Z)^4)}\\
	&\quad\le C_K\|\varphi_h\|_{ W^{2,2}_h((h\Z)^4)}
\end{align*}
and hence
\begin{equation}\label{eq:discandcontGreenareclose2}
\sup_{\substack{\varphi_h=0\text{ on }(h\Z)^4\setminus V_h\\\|\varphi_h\|_{ W^{2,2}_h((h\Z)^4)}=1}}\left\|\nabla_h^2\left(\chi^{(r)}_{h,y}\varphi_h\right)\right\|_{L^2_h((h\Z)^4))}\le C_K\,.
\end{equation}

Let us now also estimate the first factor in \eqref{eq:discandcontGreenareclose1}. The operator $T^{h,3,3,3,3-2e_i}$ preserves constant functions. Therefore for any $z$ with $|z-y|_\infty\ge\frac{r}{32}$
\begin{align}
&\left(T^{h,3,3,3,3-2e_i}\Delta\left(\eta^{(r)}_y\hat G_y^{(r)}\right)\right)(z)\nonumber\\
&\quad=\Delta\left(\eta^{(r)}_y\hat G^{(r)}_y\right)(z)+\left(T^{h,3,3,3,3-2e_i}\left(\Delta\left(\eta^{(r)}_y\hat G_y^{(r)}\right)(\cdot)-\Delta\left(\eta^{(r)}_y\hat G_y^{(r)}\right)(x)\right)\right)(z)\nonumber\\
&\quad=\Delta\left(\eta^{(r)}_y\hat G^{(r)}_y\right)(z)+O\left(h\sup_{z+(-\frac{3h}{2},\frac{3h}{2})}\left|\nabla^3\left(\eta^{(r)}_y\hat G_y^{(r)}\right)\right|\right)\nonumber\\
&\quad=\Delta\left(\eta^{(r)}_y\hat G^{(r)}_y\right)(z)+O\left(\frac{h}{r^3}\right)\,,\label{eq:discandcontGreenareclose3}
\end{align}
where we have used that $\left|T^{h,3,3,3,3-2e_i}f(z)\right|\le C\sup_{z+(-3h/2,3h/2)}|f|$ in the second step as well as the explicit formula for $\hat G^{(r)}(z,y)$ in the third step. From Lemma \ref{l:estgreenfull} and Taylor's theorem we know that for $\frac{r}{64}\le|z-y|_\infty\le 16r$
\begin{align*}
	D^h_\alpha \hat G_{h,y}^{(r)}(z)&=\partial^\alpha \hat G_y^{(r)}(z)+O\left(\frac{h}{r^{|\alpha|+1}}\right)\,,\\
	D^h_\alpha \hat G_{h,y}^{(r)}(z)&=O\left(\frac{1}{r^{|\alpha|}}\right)\,,\\
	D^h_\alpha\eta_{h,y}^{(r)}(z)&=\partial^\alpha\eta_y^{(r)}(z)+O\left(\frac{h}{r^{|\alpha|+1}}\right)\,,\\
	D^h_\alpha\eta_{h,y}^{(r)}(z)&=O\left(\frac{1}{r^{|\alpha|}}\right)\,.\\
\end{align*}
If we combine these estimates with the discrete product rule we obtain that for any $z$ with $\frac{r}{32}\le|z-y|_\infty\le 8r$
\begin{equation}\label{eq:discandcontGreenareclose4}
\Delta_h\left(\eta_{h,y}^{(r)}\hat G_{h,y}^{(r)}\right)(z)=\Delta\left(\eta_y^{(r)}\hat G_y^{(r)}\right)(z)+O\left(\frac{h}{r^3}\right)\,.
\end{equation}
Combining \eqref{eq:discandcontGreenareclose3} and \eqref{eq:discandcontGreenareclose4} we find that
\[\left|-\Delta_h\left(\eta_{h,y}^{(r)}\hat G_{h,y}^{(r)}\right)+ \sum_{i=1}^4T^{h,3,3,3,3-2e_i}\Delta\left(\eta_y^{(r)}\hat G_y\right)\right|\le C\frac{h}{r^3}\]
on $Q^h_{8r}(y)\setminus Q^h_{r/32}(y)$ and therefore
\[\left\|-\Delta_h\left(\eta_{h,y}^{(r)}\hat G_{h,y}^{(r)}\right)+ \sum_{i=1}^4T^{h,3,3,3,3-2e_i}\Delta\left(\eta_y^{(r)}\hat G_y\right)\right\|_{L^2_h(Q^h_{8r}(y)\setminus Q^h_{r/32}(y))}\le C\frac{h}{r}\,.\]
If we use this result and \eqref{eq:discandcontGreenareclose2} in \eqref{eq:discandcontGreenareclose1} we see that
\[\|H_h-\tilde H_h\|_{W^{2,2}_h((h\Z)^4)}\le C_K\frac{h}{r}\,.\]
In summary,
\begin{align*}
\|H_h-H\|_{W^{2,2}_h((h\Z)^4)}&\le \|H_h-\tilde H_h\|_{W^{2,2}_h((h\Z)^4)}+\|\tilde H_h-H\|_{W^{2,2}_h((h\Z)^4)}\le C_K\left(\frac{h^{\kappa_0}}{r^{\kappa_0}}+\frac{h}{r}\right)\le C_K\frac{h^{\kappa_0}}{r^{\kappa_0}}
\end{align*}
because $\frac{h}{r}\le1$.
Finally, Lemma \ref{l:poincaresobolev} allows us to conclude that for any $x\in(h\Z)^4$
\begin{align*}
|H_h(x)-H(x)|&\le C_K\sqrt{\log\left(2+\frac{d(x)}{h}\right)}\|H_h-H\|_{W^{2,2}_h((h\Z)^4)}\le C_K\frac{h^{\kappa_0}}{r^{\kappa_0}}\sqrt{\log\left(2+\frac{d(x)}{h}\right)}\,.
\end{align*}
This completes the proof.
\end{proof}

Before we turn to the proof of Theorem \ref{t:membranemodifiedassumptions} let us observe that Lemma \ref{l:poincaresobolev} already implies an upper bound on $G_h(x,y)$.
\begin{lemma}\label{l:estG_heasy}
For any $x,y$ we have that
\begin{equation}\label{e:estG_heasy_1}
|G_h(x,y)|\le C\sqrt{\log\left(2+\frac{d(x)}{h}\right)\log\left(2+\frac{d(y)}{h}\right)}\,.
\end{equation}
\end{lemma}
\begin{proof}
The idea is the same as in the proof of \cite[Lemma 8.1]{Muller2019}. We have
\begin{align*}
	G_h(x,y)=(G_{h,x},\delta_{h,y})_{L^2_h((h\Z)^4)}=(G_{h,x},\Delta_h^2G_{h,y})_{L^2_h((h\Z)^4)}=(\nabla_h^2G_{h,x},\nabla_h^2G_{h,y})_{L^2_h((h\Z)^4)}\,.	
\end{align*}
This implies on the one hand
\begin{equation}\label{e:estG_heasy_2}
	|G_h(x,y)|\le \|\nabla_h^2G_{h,x}\|_{L^2_h((h\Z)^4)}\|\nabla_h^2G_{h,y}\|_{L^2_h((h\Z)^4)}
\end{equation}
and on the other hand (by choosing $y=x$) that
\[|G_h(x,x)|= \|\nabla_h^2G_{h,x}\|_{L^2_h((h\Z)^4)}^2\,.\]
From Lemma \ref{l:poincaresobolev} we know that
\[|G_h(x,x)|\le \sqrt{\log\left(2+\frac{d(x)}{h}\right)}\|\nabla_h^2G_{h,x}\|_{L^2_h((h\Z)^4)}\,.\]
Combining the last two estimates we obtain
\[|G_h(x,x)|\le C\log\left(2+\frac{d(x)}{h}\right)\]
which is \eqref{e:estG_heasy_1} in the special case $x=y$. For the general case we can use \eqref{e:estG_heasy_2} to see that
\begin{align*}
|G_h(x,y)|\le \|\nabla_h^2G_{h,x}\|_{L^2_h((h\Z)^4)}\|\nabla_h^2G_{h,y}\|_{L^2_h((h\Z)^4)}\le C\sqrt{\log\left(2+\frac{d(x)}{h}\right)\log\left(2+\frac{d(y)}{h}\right)}\,.
\end{align*}
\end{proof}

Now we can turn to the proof of the main technical result of this work, Theorem \ref{t:membranemodifiedassumptions}.

\begin{proof}[Proof of Theorem \ref{t:membranemodifiedassumptions}]
Recall that according to Observation \ref{o:estimatesGh} we actually have to verify \ref{B.0'}, \ref{B.1'}, \ref{B.2'} and \ref{B.3'}.

\emph{Step 1: Proof of \ref{B.1'}}\\
Let $x,y\in (h\Z)^4$. We can assume w.l.o.g. that $d(x)\le d(y)$ (else interchange $x$ and $y$). If $d(y)< 768h$ we have that $\left|\log\left(2+\frac{\max(d(x),d(y))}{h+|x-y|}\right)\right|\le C$, and by Lemma \ref{l:estG_heasy} also $|G_h(x,y)|\le C$, so that \ref{B.1'} holds trivially. Thus we can assume $d(y)\ge768h$.

Consider first the case $|x-y|\le\frac{d(y)}{4}$. Then Lemma \ref{l:discandcontGreenareclose} with $K=2$, i.e. $r=\frac{d(y)}{2}\ge192h$ implies
\begin{align*}
&\left|G_h(x,y)-\eta^{(d(y)/2)}_{h,y}(x)\hat G^{(d(y)/2)}_h(x,y)- G(x,y)+\eta_y^{(d(y)/2)}(x)\hat G^{(d(y)/2)}(x,y)\right|\\
&\quad\le C\frac{h^{\kappa_0}}{r^{\kappa_0}}\sqrt{\log\left(2+\frac{d(x)}{h}\right)}
\end{align*}
which implies that
\begin{align*}
&\left|G_h(x,y)-\hat G^{(d(y)/2)}_h(x,y)-G(x,y)+\hat G^{(d(y)/2)}(x,y)\right|\le C\frac{h^{\kappa_0}}{r^{\kappa_0}}\sqrt{\log\left(2+\frac{2r}{h}\right)}
\,.
\end{align*}
The function $s\mapsto \frac1{s^{\kappa_0}}\sqrt{\log\left(2+2s\right)}$ is bounded on $[1,\infty)$, so that we actually obtain
\begin{equation}\label{e:mainthm1}
\left|G_h(x,y)-\hat G^{(d(y)/2)}_h(x,y)-G(x,y)+\hat G^{(d(y)/2)}(x,y)\right|\le C\,.
\end{equation}
From Lemma \ref{l:estgreenfull} we know
\[\left|\hat G_h^{(d(y)/2)}(x,y)-\frac{1}{\lambda^2}\log\left(\frac{d(y)}{|x-y|+h}\right)\right|\le C\]
(where we have absorbed a term $\frac{1}{\lambda^2}\log 2$ into the constant).
Furthermore by Lemma \ref{l:estimateGpointwise}
\[\left|G(x,y)-\hat G^{(d(y)/2)}(x,y)\right|\le C\,.\]
If we use these estimates in \eqref{e:mainthm1} we obtain
\[\left|G_h(x,y)-\frac{1}{\lambda^2}\log\left(\frac{d(y)}{|x-y|+h}\right)\right|\le C\,.\]
Because $|x-y|\le\frac{d(y)}{4}$, $\frac{d(y)}{|x-y|+h}$ is bounded away from 1 by a constant, and so
\[\left|\frac{1}{\lambda^2}\log\left(\frac{d(y)}{|x-y|+h}\right)-\frac{1}{\lambda^2}\log\left(2+\frac{d(y)}{|x-y|+h}\right)\right|\le C\,.\]
Combining this with the preceding inequality we arrive at \ref{B.1'}.

If $|x-y|\ge\frac{d(y)}{4}$ we argue similarly. We use Lemma \ref{l:discandcontGreenareclose} with $r=\frac{d(y)}{4}\ge192h$ and conclude
\[\left|G_h(x,y)-G(x,y)\right|\le C\,.\]
This combined with Lemma \ref{l:estimateGpointwise} implies again \ref{B.1'}, as now $\frac{d(y)}{|x-y|+h}$ is bounded above.

\emph{Step 2: Proof of \ref{B.2'}}\\
Recall from Lemma \ref{l:estimateGpointwise} that $a(x):=\lambda^2\lim_{\substack{(x',x'')\to(x,x)\\x'\neq x''}}(G(x',x'')-\hat G(x',x''))$ is well-defined for each $x\in (0,1)^4$ and that $a\colon(0,1)^4\to\R$ is continuous.

After this remark we can proceed similarly as in the first step. We choose $f_1(x)=a(x)$, $f_2(u,v)=\lambda^2F(u-v)$ with the $F$ from Lemma \ref{l:mangad}. Furthermore we choose $\theta_0=\frac{1}{2\kappa_0}$. Given $L$ and $\theta>\theta_0$ we take $N_0'$ so large that $768L\le |\log h|^\theta$ when $h\le\frac{1}{N_0'}$. Then $d(x)\ge h|\log h|^\theta\ge 768Lh$. 
We want to apply Lemma \ref{l:discandcontGreenareclose} with $K=8$ and $r=\frac{d(x)}{4}$ at the point $(x+hu,x+hv)$. We have that $r=\frac{d(x)}{4}\le \frac{d(x+hv)+Lh}{4}\le \frac{d(x+hv)}{2}$ and similarly $r\ge \frac{d(x+hv)}{8}$, and also $r=\frac{d(x)}{4}\ge 192Lh\ge192h$ so that all assumptions of the lemma are satisfied. We obtain
\begin{align}
	&\left|G_h(x+hu,x+hv)-\hat G^{(d(x)/4)}_h(x+hu,x+hv)-G(x+hu,x+hv)+\hat G^{(d(x)/4)}(x+hu,x+hv)\right|\nonumber\\
	&\quad\le C\frac{h^{\kappa_0}}{r^{\kappa_0}}\sqrt{\log\left(2+\frac{d(x+hu)}{h}\right)}\le C\frac{h^{\kappa_0}\sqrt{|\log h|}}{r^{\kappa_0}}\le C\frac{h^{\kappa_0}\sqrt{|\log h|}}{(h|\log h|^\theta)^{\kappa_0}}\le C|\log h|^{\frac12-\theta\kappa_0}\,.\label{e:mainthm2}
\end{align}
Here we could omit the cut-off functions $\eta_h^{(d(x)/4)}$ and $\eta^{(d(x)/4)}$ because $|x+hu-(x+hv)|\le4Lh\le\frac{d(x)}{8}$.
Since $\theta\kappa_0>\theta_0\kappa_0=\frac12$, for $N_0'$ large enough the term on the right hand side will be less than $\frac{\varepsilon}{2\lambda^2}$ whenever $h\le\frac{1}{N_0'}$.

By \eqref{e:boundGholder} in Lemma \ref{l:estimateGpointwise} we have for $u,v\in [0,L]^4$
\begin{align*}
	&\left|G(x+hu,x+hv)-\hat G^{(d(x)/4)}(x+hu,x+hv)-\frac{a(x)}{\lambda^2}-\frac{1}{\lambda^2}\log\frac{d(x)}{4}\right|\\
	&\quad\le C\left(\frac{|hu|^{\kappa_0}+|hv|^{\kappa_0}}{d(x)^{\kappa_0}}\right)\le C_L\frac{h^{\kappa_0}}{d(x)^{\kappa_0}}\le C_L|\log h|^{-\theta\kappa_0}\,.
\end{align*}
Thus we can choose $N_0'$ large enough such that for $h\le\frac{1}{N_0'}$ we have
\begin{align*}
&\sup_{u,v\in[0,L]^4\cap\Z^4}\left|G(x+hu,x+hv)-\hat G^{(d(x)/4)}(x+hu,x+hv)-\frac{a(x)}{\lambda^2}-\frac{1}{\lambda^2}\log\frac{d(x)}{4}\right|\le\frac{\varepsilon}{2\lambda^2}
\end{align*}
uniformly in $x$. Our definition of $G^{(d(x)/4)}_h$ implies that
\begin{align*}
\hat G_h^{(d(x)/4)}(x+hu,x+hv)&=F\left(\frac{x+hu}{h}-\frac{x+hv}{h}\right)-\frac{1}{\lambda^2}\log h+\frac{1}{\lambda^2}\log \frac{d(x)}{4}\\
&=F(u-v)-\frac{1}{\lambda^2}\log h+\frac{1}{\lambda^2}\log \frac{d(x)}{4}\,.
\end{align*}
Using these results in \eqref{e:mainthm2} we arrive at
\[\left|G_h(x+hu,x+hv)-F(u-v)+\frac{1}{\lambda^2}\log h-\frac{a(x)}{\lambda^2}\right|\le \frac{\varepsilon}{\lambda^2}\]
for $h\le\frac{1}{N_0'}$, which implies \ref{B.2'}.

\emph{Step 3: Proof of \ref{B.3'}}\\
This is very similar to Step 2. We set $f_3(x,y)=\lambda^2G(x,y)$, which is continuous away from the diagonal according to Lemma \ref{l:estimateGpointwise}. 

We use Lemma \ref{l:discandcontGreenareclose} with $K=L$ and $r=\frac{d(y)}{L}\le\frac1L\le|x-y|$. For $N_1'$ large enough we have $r\ge192h$, and the lemma implies
\[\left|G_h(x,y)-G(x,y)\right|\le C_L \frac{h^{\kappa_0}\sqrt{|\log h|}}{r^{\kappa_0}}\le C_L|\log h|^{\frac12-\theta\kappa_0}\]
and it suffices to take $N_1'$ so large that the right hand side is less than $\frac{\varepsilon}{\lambda^2}$ for any $h\le\frac{1}{N_1'}$.

\emph{Step 4: Proof of \ref{B.0'}}\\
Here we actually need to prove three estimates, namely
\begin{align}
\lambda^2G_h(x,x) &\le |\log h|+ C\label{e:mainthm3}\\
\lambda^2G_h(x,x) &\le C\log\left(2+\frac{d(x)}{h}\right)\label{e:mainthm4}\\
\lambda^2(G_h(x,x)-G_h(x,y))&\le \log\left(1+\frac{|x-y|}{h}\right)+C\,.\label{e:mainthm5}
\end{align}
Now \eqref{e:mainthm3} follows immediately from \ref{B.1'}, and \eqref{e:mainthm4} is a special case of Lemma \ref{l:estG_heasy}. Finally, \eqref{e:mainthm5} can be obtained from \ref{B.1'} as follows. We know that
\begin{align*}
\lambda^2(G_h(x,x)-G_h(x,y))&\le\log\left(2+\frac{d(x)}{h}\right) -\log\left(2+\frac{\max(d(x),d(y))}{h+|x-y|}\right)+C\\
&=\log\left(\frac{(d(x)+2h)(|x-y|+h)}{h(h+|x-y|+2\max(d(x),d(y)))}\right)+C
\end{align*}
so one only has to observe that
\[\frac{d(x)+2h}{h+|x-y|+2\max(d(x),d(y))}\le C\,.\]
\end{proof}

Finally we give the proof of Theorem \ref{t:convmaxmembrane}. 
\begin{proof}[{Proof of Theorem \ref{t:convmaxmembrane}}]
Because of Theorem \ref{t:logcorrconvinlaw} and Observation \ref{o:estimatesGh} all we have to check is that each of the statements \ref{A.0'}, \ref{A.1'}, \ref{A.2'}, \ref{A.3'} implies its counterpart without the prime.

We begin with \ref{A.0'}$\implies$\ref{A.0}. We know that \[\Var \varphi_{N, v} \le \min\left(\log N+ \alpha_0',\alpha_0'\log (2+d_N(v))\right)\]and this implies in particular that
\[\Var \varphi_{N, v} \le \log N+ \alpha_0'\,.\]
Furthermore, if we know
\[\Var \varphi_{N, v}-\Cov(\varphi_{N, v},\varphi_{N, u})  \le \log_+ |u-v| + 2\alpha_0'\]
then by symmetry this also holds with $u,v$ interchanged, so that we actually have
\begin{align*}
\max\left(\Var \varphi_{N, v}-\Cov(\varphi_{N, v},\varphi_{N, u}),\Var \varphi_{N, u}-\Cov(\varphi_{N, v},\varphi_{N, u})\right)\le \log_+ |u-v| + 2\alpha_0'
\end{align*}
and a short calculation shows that this is the same as
\[\E (\varphi_{N, v} - \varphi_{N, u})^2  \le 2 \log_+ |u-v| - | \Var \varphi_{N, v} - \Var \varphi_{N, u}| + C\,.\]

For \ref{A.1'}$\implies$\ref{A.1} one has to verify that $\min(d(u),d(v))\ge\delta N$ implies 
\[\left|\log\left(2+\frac{\max(d_N(u),d_N(v))}{1+|u-v|}\right)-\log\left(\frac{N}{1+|u-v|}\right)\right|\le C_\delta\,,\]
which is straightforward.

For \ref{A.2'}$\implies$\ref{A.2} we fix some $\theta>\theta_0$. Given $L,\varepsilon,\delta$, we choose $N_0\ge N_0'(L,\varepsilon,\theta)$ large enough such that $|\log N|^\theta\le\delta N$ for all $N\ge N_0$ and conclude \ref{A.2}. Analogously one sees that \ref{A.3'}$\implies$\ref{A.3}.

\end{proof}

\paragraph{Acknowledgements}
The author would like to thank Stefan M\"uller for some valuable suggestions that helped improve the argument. He would also like to thank Ofer Zeitouni for inspiring discussions and his encouragement.
 
Most of this work was conducted while the author was visiting the Courant Institute of NYU while supported by the Global Math Network, and he would like to thank the Institute for its hospitality.

The author was supported by the Hausdorff Center for Mathematics (GZ
2047/1, Projekt-ID 3906$ $85813) through the Bonn International Graduate School of Mathematics, and by the German National Academic Foundation.

\bibliographystyle{alpha_edited2}
\bibliography{M4DM_AOP_final}

\end{document}